\documentclass[reqno,11pt]{amsart} 
\numberwithin{equation}{section}
\usepackage{amssymb}
\usepackage{xy}
\xyoption{all}

\theoremstyle{plain} 
\newtheorem{theorem}{\bf Theorem}[section]
\newtheorem{lemma}[theorem]{\bf Lemma}
\newtheorem{corollary}[theorem]{\bf Corollary}
\newtheorem{proposition}[theorem]{\bf Proposition}

\theoremstyle{definition} 
\newtheorem{definition}[theorem]{\bf Definition}
\newtheorem{remark}[theorem]{\bf Remark}

\begin{document}

\title[Gorenstein  derived functors from special rings to arbitrary rings]{Gorenstein  derived functors from special rings to arbitrary rings} 

\author[T. Zhao]{Tiwei Zhao} 


%
\thanks{ 
Email: tiweizhao@hotmail.com
}


\maketitle

\begin{abstract}
In this note, we mainly extend some Gorenstein homological properties from special rings (Noetherian or coherent rings ) to arbitrary rings by introducing the notions of Gorenstein  weak injective and weak projective modules respectively.
\end{abstract}

\section{Introduction and Preliminaries} 

Throughout $R$ is an associative ring with identity and all modules are unitary. Unless stated otherwise, an $R$-module
will be understood to be a left $R$-module. As usual, $pd_R(M)$ and
$id_R(M)$  will denote the projective and injective dimensions of an $R$-module $M$, respectively. We also denote  by $\mathcal{I}$ and $\mathcal{P}$ the classes of injective and projective $R$-modules, respectively. For unexplained concepts and notations, we refer the
readers to \cite{EJ1,Ro}.

In 1970, in order to generalize the homological properties from Noetherian rings to coherent rings, Stenstr\"{o}m introduced the notion of FP-injective modules  in \cite{St}. In this process, finitely generated modules should in general be replaced by finitely presented modules.  Recall  that an $R$-module $M$ is called FP-injective if $\mbox{Ext}^1_R(N,M)=0$ for any finitely presented $R$-module $N$, and accordingly, the FP-injective dimension of $M$, denoted by $FP\mbox{-}id_R(M)$, is defined to be the smallest non-negative integer $n$ such that $\mbox{Ext}^{n+1}_R(N,M)=0$ for any finitely presented $R$-module $N$. Recently, as  extending work of Stenstr\"{o}m's viewpoint, Gao and Wang introduced the notion of weak injective (and weak flat) modules (\cite{GW}). In this process, finitely presented modules were  replaced by super finitely
presented modules (see \cite{GW0} or the seventh paragraph in this section for the definition). It was shown that many results of a homological nature may be generalized from coherent rings to arbitrary
rings (see \cite{GH,GW} for details).

In 1965, Eilenberg and Moore first introduced the
viewpoint of relative homological algebra in \cite{EM}. Since then the relative homological algebra,
especially the Gorenstein homological algebra, got a rapid
development. The main ideal of Gorenstein homological algebra is to replace projective and injective modules by Gorenstein projective and injective modules. Recall from \cite{EJ1} that an $R$-module $M$ is called Gorenstein injective  if there exists an exact sequence of injective $R$-modules
$$
 \xymatrix@C=0.5cm{
  \cdots \ar[r] & I_1 \ar[r]^{} & I_0 \ar[r]^{} & I^0 \ar[r]^{} & I^1 \ar[r] & \cdots }
$$
such that $M=\operatorname{Coker}(I_1\rightarrow I_0)$ and the functor $\operatorname{Hom}_R(I,-)$ leaves this sequence exact whenever $I$ is an injective $R$-module. Dually, one may give the definition of Gorenstein projective $R$-module. Nowadays, it has been developed to an advanced level (e.g. \cite{AM,Ch,CFH,EJ,EJ1,Ho,Ho2,Hu,SSW,ZAD}).
However, in the most results of Gorenstein homological algebra, the
condition `noetherian' is essential. In order to make the
similar properties of Gorenstein homological algebra hold in a wider
environment, Ding  and his coauthors introduced the
notions of Gorenstein FP-injective and strongly Gorenstein flat
modules (see \cite{DLM,MD} for details). Later on, Gillespie renamed Gorenstein FP-injective modules as
Ding injective modules, and strongly Gorenstein flat modules as Ding
projective modules (\cite{Gi}). Recall from \cite{Gi} that an $R$-module $M$ is called Ding injective (or Gorenstein FP-injective in the sense of \cite{MD}) if there exists an exact sequence of injective $R$-modules
$$
 \xymatrix@C=0.5cm{
  \cdots \ar[r] & I_1 \ar[r]^{} & I_0 \ar[r]^{} & I^0 \ar[r]^{} & I^1 \ar[r] & \cdots }
$$
such that $M=\operatorname{Coker}(I_1\rightarrow I_0)$ and the functor $\operatorname{Hom}_R(E,-)$ leaves this sequence exact whenever $E$ is an FP-injective $R$-module. It was shown that the Ding homological algebra over coherent rings possess many nice
properties analogous to the Gorenstein homological algebra over Noetherian rings (see \cite{DLM,Gi,MD,Ya} for details). It is natural to ask how generalize the Gorenstein homological algebra from coherent rings to arbitrary rings. Notice that the definition of Gorenstein injective modules is based on the injective modules, and the definition of Ding injective modules is based on the FP-injective modules.  So, in order to make the properties of relative homological algebra hold over any ring, it seems that we have to replace FP-injective modules by weak injective modules.

The main purpose of this paper is to generalize the Gorenstein homological properties from Noetherian or coherent rings to any rings.
In Section 2, we introduce the notion of Gorenstein weak injective modules in terms of weak injective modules,  and discuss some of the properties of these modules. We show that every Gorenstein weak injective $R$-module is either injective or has  weak injective dimension $\infty$ (see Proposition $\ref{either or}$); and  an $R$-module $M$ is Gorenstein weak injective if and only if $M$ has an exact left $\mathcal{WI}$-resolution and $\operatorname{Ext}^i_R(W,M)=0$ for any weak injective $R$-module $W$ and any $i\geq 1$ (see Proposition $\ref{prop2.7}$). We also show that the class of Gorenstein weak injective modules is injectively resolving (see Proposition $\ref{injectively coresolving}$). Then we introduce and  study the Gorenstein weak injective dimension of modules. Moreover, the existence of Gorenstein weak injective preenvelope is given (see Proposition $\ref{Gorenstein weak injective preenvelope}$). Finally, we introduce the notion of $\mathcal{GWI}$-copure exact sequence, and further give characterizations of Gorenstein weak injective modules in terms of it (see Propositions $\ref{prop5.2}$ and $\ref{prop5.3}$).
In Section 3, we introduce and investigate dually  Gorenstein weak projective modules in terms of weak flat modules.
In Section 4, we mainly investigate a class of relative right derived functors, denoted by ${\operatorname{Ext}}^n_{\mathcal{GW}}(-,-)$, with respect to Gorenstein weak  modules, and study the balance of this functor (see Theorem $\ref{balanced functor}$). Moreover, we give some characterizations of Gorenstein weak injective and  projective dimensions in terms of this derived functor (see Proposition $\ref{characterizations derived functors}$).
In Section 5, we continue to investigate
another derived functor, denoted by $\widehat{\operatorname{Ext}}^n_{\mathcal{GW}}(-,-)$. We show that the functor $\widehat{\operatorname{Ext}}^n_{\mathcal{GW}}(-,-)$ actually measures the distance between  the usual right derived functor ${\operatorname{Ext}}^n_{R}(-,-)$ and the Gorenstein weak  right derived functor ${\operatorname{Ext}}^n_{\mathcal{GW}}(-,-)$. Finally, we give  the balance of this functor (see Theorem $\ref{Tate balance}$).

In the following,  we  recall some terminologies and  preliminaries. For more details, we refer the readers to \cite{EJ1,GH,GW}.

\begin{definition} (\cite{EJ1}) Let $\mathcal {F}$ be a class of $R$-modules. By an
\emph{$\mathcal {F}$-preenvelope} of an $R$-module $M$, we mean a morphism
$\varphi: M \rightarrow F$ where $F\in \mathcal {F}$ such that for
any morphism $f: M\rightarrow F^{'}$ with $F^{'}\in \mathcal {F}$,
there exists a morphism $g:F\rightarrow F^{'}$ such that
$g\varphi=f$, that is, there is the following commutative diagram:
$$
\xymatrix{
  M \ar[dr]_{f} \ar[r]^{\varphi}
                & F \ar@{.>}[d]^{g}  \\
                & F'             }
$$
 If furthermore, when $F^{'}=F$ and $f=\varphi$,
the only such $g$ are automorphisms of $F$, then $\varphi: M
\rightarrow F$ is called an \emph{$\mathcal {F}$-envelope} of $M$.

Dually, one may give the notion of \emph{$\mathcal {F}$-(pre)cover} of an $R$-module.
Note that $\mathcal {F}$-envelopes and $\mathcal {F}$-covers may not
exist in general, but if they exist, they are unique up to
isomorphism.\end{definition}

In the process that some results of a homological nature may be generalized from coherent rings to arbitrary rings, the notion of super finitely presented modules plays a crucial role. Recall from \cite{GW0} that an $R$-module $M$ is said to be \emph{super finitely presented} (or $FP_\infty$ in \cite{HM}) if there exists an exact sequence $\cdots \rightarrow F_1\rightarrow F_0\rightarrow M\rightarrow 0$, where each $F_i$ is finitely generated and projective. Then Gao and Wang gave the definition of weak injective (and weak flat) modules in terms of super finitely presented modules in \cite{GW}, which is a generalization of the notion of FP-injective modules.

\begin{definition}(\cite{GW})
An $R$-module $M$ is called \emph{weak injective} if $\mbox{Ext}^1_R(N,M)=0$ for any super finitely presented $R$-module $N$. A right $R$-module $M$ is called \emph{weak flat} if $\mbox{Tor}_1^R(M,N)=0$ for any super finitely presented $R$-module $N$.

Accordingly, the \emph{weak injective dimension} of an $R$-module $M$, denoted by $wid_R(M)$, is defined to be the smallest non-negative integer $n$ such that $\mbox{Ext}^{n+1}_R(N,M)=0$ for any super finitely presented $R$-module $N$, and the \emph{weak flat dimension} of a right $R$-module $M$, denoted by $wfd_R(M)$, is defined to be the smallest non-negative integer $n$ such that $\mbox{Tor}_{n+1}^R(M,N)=0$ for any super finitely presented $R$-module $N$.
\end{definition}

We denote by $\mathcal{WI}$ and $\mathcal{WF}$ the classes of  weak injective and weak flat $R$-modules, respectively.  By  \cite[Thm. 3.4]{GH},  every $R$-module  has  a weak injective preenvelope. So for any $R$-module $M$, $M$ has a right $\mathcal{WI}$-resolution, that is, there exists a $\mbox{Hom}_R(-,\mathcal{WI})$-exact complex
$$\xymatrix@C=0.5cm{
  0 \ar[r]&  M\ar[r]& E^0\ar[r]& E^1\ar[r]&E^2\ar[r]& \cdots},$$
  where each $E^i$ is weak injective.
Moreover, since every injective $R$-module is weak injective, this complex is also exact. On the other hand, every $R$-module has a weak injective cover by \cite[Thm. 3.1]{GH}. So every $R$-module $M$ has a left  $\mathcal{WI}$-resolution, that is, there exists a $\mbox{Hom}_R(\mathcal{WI},-)$-exact complex
$$\xymatrix@C=0.5cm{
 \cdots \ar[r]&  W_2\ar[r]& W_1\ar[r]& W_0\ar[r]&M\ar[r]& 0},$$
  where each $W_i$ is weak injective. But this complex is not necessarily exact.

Since every $R$-module has a weak flat cover by \cite{GH}, every $R$-module  has a left $\mathcal{WF}$-resolution. So for any $R$-module $M$,  there exists a $\mbox{Hom}_R(\mathcal{WF},-)$-exact complex
$$\xymatrix@C=0.5cm{
 \cdots \ar[r]&  F_2\ar[r]& F_1\ar[r]& F_0\ar[r]&M\ar[r]& 0},$$
  where each $F_i$ is weak flat.
Moreover, since every projective $R$-module is weak flat, this complex is also exact. On the other hand, every $R$-module  has  a weak flat preenvelope by \cite[Thm. 2.15]{GW}. So for any $R$-module $M$, $M$ has a right $\mathcal{WF}$-resolution, that is, there exists a $\mbox{Hom}_R(-,\mathcal{WF})$-exact complex
$$\xymatrix@C=0.5cm{
  0 \ar[r]&  M\ar[r]& F^0\ar[r]& F^1\ar[r]&F^2\ar[r]& \cdots},$$
  where each $F^i$ is weak flat. Similarly, this complex is not necessarily exact.

\section{Gorenstein weak injective modules and dimension}

In this section, we give the definition of Gorenstein weak injective  modules in terms of weak injective  modules,  and discuss some of the properties of these modules.

\begin{definition}\label{Gorenstein weak injective}
An $R$-module $M$ is called \emph{Gorenstein weak injective} if there exists an exact sequence of  injective $R$-modules
$$
 \emph{\textbf{I}}=\xymatrix@C=0.5cm{
  \cdots \ar[r] & I_1 \ar[r]^{} & I_0 \ar[r]^{} & I^0 \ar[r]^{} & I^1 \ar[r] & \cdots }
$$
such that $M=\operatorname{Coker}(I_1\rightarrow I_0)$ and the functor $\operatorname{Hom}_R(W,-)$ leaves this sequence exact whenever $W$ is a weak injective $R$-module.
\end{definition}

We will denote by $\mathcal{GWI}$ the class of Gorenstein weak injective $R$-modules.

\begin{remark}\label{1-4}

$(1)$ Every injective $R$-module is Gorenstein weak injective.

$(2)$ Since every FP-injective $R$-module is weak injective, every Gorenstein weak injective is Ding injective $R$-module (in the sense of \cite{Gi}). If $R$ is a left coherent ring, then the class of  Gorenstein weak injective $R$-modules coincides with the class of Ding injective $R$-modules. Moreover, we have the following implications:
$$
\begin{array}{ccc}
\mbox{Gorenstein weak injective }R\mbox{-modules}&\Rightarrow&\mbox{Ding injective }R\mbox{-modules}\\
&\Rightarrow&\mbox{Gorenstein injective }R\mbox{-modules}.
\end{array}
$$
If $R$ is a left Noetherian ring, then these three kinds of $R$-modules coincide.

$(3)$ The class of Gorenstein weak injective $R$-modules is closed under direct products.

$(4)$ If $
 \emph{\textbf{I}}=\xymatrix@C=0.5cm{
  \cdots \ar[r] & I_1 \ar[r]^{} & I_0 \ar[r]^{} & I^0 \ar[r]^{} & I^1 \ar[r] & \cdots }
$ is an exact sequence of  injective $R$-modules such that the functor $\operatorname{Hom}_R(W,-)$ leaves this sequence exact whenever $W$ is a weak injective $R$-module, then by symmetry, all the images, the kernels and the cokernels of $\emph{\textbf{I}}$ are Gorenstein weak injective.
\end{remark}

\begin{lemma}\label{lemma1}
Let $M$ be a Gorenstein weak injective $R$-module. Then $\operatorname{Ext}^i_R(W,M)=0$ for any weak injective $R$-module $W$ and any $i\geq 1$.
\end{lemma}

\begin{proof}
By the definition of Gorenstein weak injective $R$-modules, $M$ admits an injective resolution $0\rightarrow M\rightarrow I^0\rightarrow I^1\rightarrow \cdots$ which remains exact after applying the functor $\mbox{Hom}_R(W,-)$ for any weak injective $R$-module $W$. So $\operatorname{Ext}^i_R(W,M)=0$ for any $i\geq 1$.
\end{proof}

\begin{proposition}\label{either or}
A Gorenstein weak injective $R$-module is either injective or has  weak injective dimension $\infty$. Consequently, $\mathcal{GWI}\bigcap \widetilde{\mathcal{WI}}=\mathcal{I}$, where $\widetilde{\mathcal{WI}}$ denote the class of $R$-modules with finite weak injective dimension.
\end{proposition}

\begin{proof}
Let $M$ be a Gorenstein weak injective $R$-module and assume that $wid_R(M)=n<\infty$, that is, $\mbox{Ext}^{n+1}_R(N,M)=0$ for  any super finitely presented $R$-module $N$. Choose a partial injective resolution of $M$: $0\rightarrow M\rightarrow I^0\rightarrow I^1\rightarrow \cdots \rightarrow I^{n-1}\rightarrow V^n\rightarrow 0$. It follows from the isomorphism $\mbox{Ext}^1_R(N,V^n)\cong \mbox{Ext}^{n+1}_R(N,M)$ that $V^n$ is weak injective. Now  $0\rightarrow M\rightarrow I^0\rightarrow I^1\rightarrow \cdots \rightarrow I^{n-1}\rightarrow V^n\rightarrow 0$ represents an element of $\mbox{Ext}^n_R(V^n,M)$, but this group equals 0 by Lemma $\ref{lemma1}$. Thus this sequence is split exact, and so $M$ is injective as a direct summand of $I^0$.
\end{proof}

\begin{corollary}\label{iff}
An $R$-module is injective if and only if it is weak injective and Gorenstein weak injective.
\end{corollary}

\begin{corollary}
If the class of Gorenstein weak injective $R$-modules is closed under direct sums, then the ring $R$ is left Noetherian.
\end{corollary}

\begin{proof}
Note that the class of weak injective $R$-modules is closed under direct sums by \cite[Prop. 2.3]{GW}, and every injective $R$-module is weak injective. So a direct sum of injective $R$-modules is weak injective. Moreover,  it is also Gorenstein weak injective by hypothesis and Remark $\ref{1-4}$(1). It follows then from Corollary $\ref{iff}$ that  a direct sum of injective $R$-modules must be injective. Thus $R$ is left Noetherian.
\end{proof}

In general, weak injective modules need not be Gorenstein weak injective as shown by the following proposition.

\begin{proposition}\label{Noetherian}
A ring $R$ is left Noetherian if and only if every weak injective module is Gorenstein weak injective.
\end{proposition}

\begin{proof}
$\Rightarrow$. It follows from the fact that the class of weak injective $R$-modules coincide with the class of injective $R$-modules over a Noetherian ring $R$.

$\Leftarrow$. Let $W$ be a weak injective $R$-module and $0\rightarrow W\rightarrow E \rightarrow N\rightarrow 0$ an exact sequence with $E$ injective. It is easy to verify that $N$ is also weak injective, and hence $\mbox{Ext}^1_R(N,W)=0$ by hypothesis and Lemma $\ref{lemma1}$. That is, the sequence $0\rightarrow W\rightarrow E \rightarrow N\rightarrow 0$ is split. Thus $W$ is injective, and $R$ is left Noetherian.
\end{proof}

As what said in \cite[Sec. 4]{GH}, $\mbox{Hom}_R(-,-)$ is left balanced on $_R\mathcal{M}\times   _R\!\mathcal{M}$ by $\mathcal{WI}\times \mathcal{WI}$, where $_R\mathcal{M}$ denotes the category of left $R$-modules. Denote by $\mbox{Ext}_i^{\mathcal{WI}}(-,-)$ the $i$th left derived functor of $\mbox{Hom}_R(-,-)$ with respect to $\mathcal{WI}\times \mathcal{WI}$. For any $R$-modules $M$ and $N$,  $\mbox{Ext}_i^{\mathcal{WI}}(M,N)$ can be computed by using a right $\mathcal{WI}$-resolution of $M$ or a left $\mathcal{WI}$-resolution of $N$. That is, let $\cdots\rightarrow W_2\rightarrow W_1\rightarrow W_0\rightarrow N\rightarrow 0$ be a left $\mathcal{WI}$-resolution of $N$. Applying the functor $\mbox{Hom}_R(M,-)$ to it, we have the deleted complex $$\cdots\rightarrow \mbox{Hom}_R(M,W_2)\rightarrow \mbox{Hom}_R(M,W_1)\rightarrow \mbox{Hom}_R(M,W_0)\rightarrow  0.$$ Then $\mbox{Ext}_i^R(M,N)$ is exactly the $i$th homology of the above complex. Now there exists a canonical homomorphism $\sigma:\mbox{Ext}_0^{\mathcal{WI}}(M,N)\rightarrow \mbox{Hom}_R(M,N)$. Let $\overline{\operatorname{Ext}}_0^{\mathcal{WI}}(W,M)=\mbox{Ker}\sigma$ and $\overline{\operatorname{Ext}}^0_{\mathcal{WI}}(W,M)=\mbox{Coker}\sigma$.

We next give some characterizations of Gorenstein weak injective modules.

\begin{proposition}\label{prop2.7}
The following are equivalent for an $R$-module $M$:

$(1)$ $M$  is Gorenstein weak injective;

$(2)$ $\operatorname{Ext}^i_R(W,M)=0=\operatorname{Ext}_i^{\mathcal{WI}}(W,M)$ for any $i\geq 1$ and $\overline{\operatorname{Ext}}^0_{\mathcal{WI}}(W,M)=0=\overline{\operatorname{Ext}}_0^{\mathcal{WI}}(W,M)$ for any projective or weak injective $R$-module $W$;

$(3)$ $\operatorname{Ext}^i_R(\widetilde{W},M)=0=\operatorname{Ext}_i^{\mathcal{WI}}(\widetilde{W},M)$ for any $i\geq 1$ and $\overline{\operatorname{Ext}}^0_{\mathcal{WI}}(\widetilde{W},M)=0=\overline{\operatorname{Ext}}_0^{\mathcal{WI}}(\widetilde{W},M)$ for any $R$-module $\widetilde{W}$ with $pd_R(\widetilde{W})<\infty$ or $wid_R(\widetilde{W})<\infty$;

$(4)$ $M$ has an exact left $\mathcal{WI}$-resolution and $\operatorname{Ext}^i_R(W,M)=0$ for any weak injective $R$-module $W$ and any $i\geq 1$;

$(5)$ $M$ has an exact left $\mathcal{WI}$-resolution and $\operatorname{Ext}^i_R(\widetilde{W},M)=0$ for any $R$-module $\widetilde{W}$ with $wid_R(\widetilde{W})<\infty$ and any $i\geq 1$.

Moreover, if $R$ satisfies $wid_R(R)<\infty$ as a left $R$-module, then the above conditions are equivalent to

$(6)$ $\operatorname{Ext}^i_R(W,M)=0$ for any weak injective $R$-module $W$ and any $i\geq 1$.
\end{proposition}

\begin{proof}
(1) $\Rightarrow$  (3). We use induction on $n=wid_R(\widetilde{W})<\infty$. If $n=0$, then the assertions follow from the definition of the functor $\mbox{Ext}^{\mathcal{WI}}_i(-,-)$ and Lemma $\ref{lemma1}$. Suppose that the results hold for the case $n-1$. Consider an exact sequence $0\rightarrow \widetilde{W}\rightarrow W\rightarrow V\rightarrow 0$ with $W$ weak injective. Then $wid_R(V)=n-1$. Assume that the sequence $\emph{\textbf{I}}$ is defined as in Definition $\ref{Gorenstein weak injective}$.
 Since each term of $\emph{\textbf{I}}$ is injective, we have the following exact sequence of complexes
$$
0\rightarrow \mbox{Hom}_R(V,\emph{\textbf{I}})\rightarrow\mbox{Hom}_R(W,\emph{\textbf{I}})\rightarrow \mbox{Hom}_R(\widetilde{W},\emph{\textbf{I}})\rightarrow 0.
$$
Note that the complex $\mbox{Hom}_R(W,\emph{\textbf{I}})$  is exact by Definition $\ref{Gorenstein weak injective}$, and the complex $\mbox{Hom}_R(V,\emph{\textbf{I}})$ is  exact by the induction hypothesis. So the complex  $\mbox{Hom}_R(\widetilde{W},\emph{\textbf{I}})$ is also exact.
It is easy to verify that (3) holds. By a similar argument, we may also get that (3) holds for the case $pd_R(\widetilde{W})<\infty$.

(3) $\Rightarrow$ (2) is trivial.

(2) $\Rightarrow$ (4). Since every $R$-module has a weak injective cover, one easily get that every $R$-module has a left $\mathcal{WI}$-resolution.  Let $\cdots\rightarrow W_2\rightarrow W_1\rightarrow W_0\rightarrow M\rightarrow 0$ be a left $\mathcal{WI}$-resolution of $M$. If we set $W=R$, we can easily get that it is also exact by assumption.

(4) $\Rightarrow$ (5) holds by dimension shifting.

(5) $\Rightarrow$ (1). Let $f:W_0\rightarrow M$ be a weak injective cover of $M$. Consider an exact sequence $(\xi): \xymatrix@C=0.5cm{
  0 \ar[r] & W_0 \ar[r]^{i} & E \ar[r]^{} & N \ar[r] & 0 }$ with $E$  injective. It is easy to verify that $N$ is  weak injective. So $\operatorname{Ext}^i_R(N,M)=0$ by hypothesis, and hence  we have the following exact sequence $$\xymatrix@C=0.5cm{\mbox{Hom}_R(E,M)\ar[r]^{i^*} &\mbox{Hom}_R(W_0,M)\ar[r]& 0}$$ by applying the functor $\mbox{Hom}_R(-,M)$ to the sequence $(\xi)$. Thus, for $f:W_0\rightarrow M$, there exists $g:E\rightarrow M$ such that $gi=i^*(g)=f$. Moreover, since $f$ is a weak injective cover, there exists $h:E\rightarrow W_0$ such that $fh=g$. Hence we have $fhi=f$ and thus $hi$ is an isomorphism. This implies that $W_0$ is injective as a direct summand of $E$. Since $f$ is a weak injective cover, we have an exact sequence $\mbox{Hom}_R(W,W_0)\rightarrow \mbox{Hom}_R(W,\mbox{Im}f)\rightarrow 0$ for any weak injective $R$-module $W$. Moreover, from the following exact sequence
  $$
  \mbox{Hom}_R(W,W_0)\rightarrow \mbox{Hom}_R(W,\mbox{Im}f)\rightarrow \mbox{Ext}^1_R(W,\mbox{Ker}f)\rightarrow \mbox{Ext}^1_R(W,W_0)=0
  $$
  we obtain that $\mbox{Ext}^1_R(W,\mbox{Ker}f)=0$.  We repeat the argument by replacing $M$ with $\mbox{Ker}f$ to get a weak injective cover $f_1:W_1\rightarrow \mbox{Ker}f$ and $W_1$ is injective. Continue this process, we may obtain a complex $(\varrho): \cdots\rightarrow W_1\rightarrow W_0\rightarrow M\rightarrow 0$, where each $W_i$ is weak injective, such that  the functor $\operatorname{Hom}_R(W,-)$ leaves this sequence exact whenever $W$ is a weak injective $R$-module. Moreover,  we have $\operatorname{Ext}_i^{\mathcal{WI}}(R,M)=0$ for any $i\geq1$ and $\operatorname{Ext}_0^{\mathcal{WI}}(R,M)\cong M$ by hypothesis. It follows then that the complex $(\varrho)$ is exact. On the other hand, we take an injective resolution of $M$: $0\rightarrow M\rightarrow I^0\rightarrow I^1\rightarrow \cdots$. Since $\operatorname{Ext}^i_R(W,M)=0$ for any weak injective $R$-module $W$ and any $i\geq 1$, it is $\mbox{Hom}_R(\mathcal{WI},-)$-exact. Assembling this sequence with the sequence $(\varrho)$, we get the desired sequence as in Definition $\ref{Gorenstein weak injective}$, and hence $M$  is Gorenstein weak injective.

(1) $\Rightarrow$ (6) follows from Lemma $\ref{lemma1}$.

(6) $\Rightarrow$ (1).  As a similar argument to the proof of (5) $\Rightarrow$ (1), we have a $\mbox{Hom}_R(\mathcal{WI},-)$-exact complex $$
 \emph{\textbf{I}}=\xymatrix@C=0.5cm{
  \cdots \ar[r] & I_1 \ar[r]^{} & I_0 \ar[r]^{} & I^0 \ar[r]^{} & I^1 \ar[r] & \cdots }
$$
of injective $R$-modules with $M=\operatorname{Coker}(I_1\rightarrow I_0)$. By induction, it is easy to verify  that the complex $\mbox{Hom}_R(\widetilde{W},\emph{\textbf{I}})$ is exact for any $R$-module with $wid_R(\widetilde{W})=n<\infty$.  In particular, $\mbox{Hom}_R(R,\emph{\textbf{I}})$ is  exact by assumption, and thus $\emph{\textbf{I}}$ is exact. Therefore, $M$ is Gorenstein weak injective.
\end{proof}

\begin{remark}
Following the proof of (5) $\Rightarrow$ (1) in Proposition $\ref{prop2.7}$, the kernel of a weak injective cover of any Gorenstein weak injective $R$-module is Gorenstein weak injective.
\end{remark}

\begin{corollary}\label{coro2.9}
An $R$-module $M$ is Gorenstein weak injective if and only if there is an exact sequence $0\rightarrow L\rightarrow I\rightarrow M\rightarrow 0$ with $I$  injective and $L$ Gorenstein weak injective.
\end{corollary}

\begin{proof}
The necessity is clear. For the sufficiency, consider the exact sequence $(\diamondsuit): 0\rightarrow L\rightarrow I\rightarrow M\rightarrow 0$ with $I$  injective and $L$ Gorenstein weak injective. Since $\mbox{Ext}^i_R(W,L)=0$ for any weak injective $R$-module $W$ and any $i\geq 1$ by Lemma $\ref{lemma1}$, it is easy verify that $\mbox{Ext}^i_R(W,M)=0$ for any weak injective $R$-module $W$ and the sequence $(\diamondsuit)$ is $\mbox{Hom}_R(\mathcal{WI},-)$-exact. Since $L$ is Gorenstein weak injective, by the definition, there is a $\mbox{Hom}_R(\mathcal{WI},-)$-exact exact sequence $\cdots\rightarrow I_2\rightarrow I_1\rightarrow L\rightarrow 0$ with each $I_i$ injective. Assembling this sequence with the sequence $(\diamondsuit)$, we have the following commutative diagram:
$$\xymatrix@=0.4cm{
  \cdots \ar[rr] && I_2 \ar[rr]^{} &&I_1 \ar[rd] \ar[rr]^{} && I\ar[rr]&&M \ar[rr]^{}  && 0\\
  &&&&&L\ar[ru]\ar[rd]&&&&&\\
  &&&&0\ar[ru]&&0&&&&}$$
which shows that $M$ is Gorenstein weak injective by Proposition $\ref{prop2.7}$.
\end{proof}

\begin{proposition}\label{extension closed}
Given an exact sequence $0\rightarrow L\rightarrow M\rightarrow N\rightarrow 0$. If $L$ and $N$ are Gorenstein weak injective, then so is $M$. That is, the class of Gorenstein weak injective $R$-modules is closed under extensions.
\end{proposition}

\begin{proof}
Since $L$ and $N$ are Gorenstein weak injective, we have that $\mbox{Ext}^i_R(W,L)=0=\mbox{Ext}^i_R(W,N)$ for any weak injective $R$-module $W$ and any $i\geq 1$ by Lemma $\ref{lemma1}$. It is easy to verify that $\mbox{Ext}^i_R(W,M)=0$ for any weak injective $R$-module $W$  and any $i\geq 1$. So, to prove that $M$ is Gorenstein weak injective, it suffices to show that $M$ has an exact left $\mathcal{WI}$-resolution by Proposition $\ref{prop2.7}$. By hypothesis, there exists the following exact left $\mathcal{WI}$-resolution of $L$ and $N$:
\begin{gather*}
   \emph{\textbf{I}}'= \xymatrix@C=0.5cm{
  \cdots \ar[r] & I_1' \ar[r]^{d_1'} &I_0' \ar[r]^{d_0'} & L \ar[r]^{}  & 0} \\
   \emph{\textbf{I}}''= \xymatrix@C=0.5cm{
  \cdots \ar[r] & I_1'' \ar[r]^{d_1''} &I_0'' \ar[r]^{d_0''} & N \ar[r]^{}  & 0},
  \end{gather*}
where all $I_i'$ and $I_i''$ are injective, and all kernels of $\emph{\textbf{I}}'$ and $\emph{\textbf{I}}''$ are Gorenstein weak injective (such sequences exist by the definition of Gorenstein weak injective). Consider the following diagram:
 $$
 \xymatrix{
 0\ar[r]&I_0'\ar[r]^{\!\!{1\choose 0}}\ar[d]^{d_0'}&I_0'\oplus I_0''\ar[r]^{\ (0\ 1)}&I_0''\ar[r]\ar[d]^{d_0''}&0\\
 0\ar[r]&L\ar[r]^{f}\ar[d]&M\ar[r]^{g}&N\ar[r]\ar[d]&0\\
 &0&&0&
 }
 $$
 Since $L$ is Gorenstein weak injective, $\mbox{Ext}^1_R(I_0'',L)=0$, and thus we have the following exact sequence
 $$
 \xymatrix@C=0.5cm{
   0 \ar[r] & \mbox{Hom}_R(I_0'',L) \ar[r]^{f_*} & \mbox{Hom}_R(I_0'',M) \ar[r]^{g_*} & \mbox{Hom}_R(I_0'',N) \ar[r] & 0 }.
 $$
 Since $g_*$ is epimorphic, there exists $\alpha:I_0''\rightarrow M$ such that $d_0''=g_*(\alpha)=g\alpha$. For any $(e_0',e_0'')\in I_0'\oplus I_0''$, we define  $d_0:I_0'\oplus I_0''\rightarrow M$ by $d_0(e_0',e_0'')=fd_0'(e_0')+\alpha(e_0'')$. Then it is easy to verify that $d_0$ makes the above diagram commute. By Snake Lemma, we have the following commutative diagram:
 $$
 \xymatrix{
 &0\ar[d]&0\ar[d]&0\ar[d]&\\
 0\ar[r]&\mbox{Ker}d_0'\ar[r]\ar[d]&\mbox{Ker}d_0\ar[r]\ar[d]&\mbox{Ker}d_0''\ar[r]\ar[d]&0\\
 0\ar[r]&I_0'\ar[r]^{\!\!{1\choose 0}}\ar[d]^{d_0'}&I_0'\oplus I_0''\ar[r]^{\ (0\ 1)}\ar[d]^{d_0}&I_0''\ar[r]\ar[d]^{d_0''}&0\\
 0\ar[r]&L\ar[r]^{f}\ar[d]&M\ar[r]^{g}\ar[d]&N\ar[r]\ar[d]&0\\
 &0&0&0&
 }
 $$
Since $\mbox{Ker}d_0'$ and $\mbox{Ker}d_0''$ are Gorenstein weak injective,  $\mbox{Ext}^i_R(W,\mbox{Ker}d_0')=0=\mbox{Ext}^i_R(W,\mbox{Ker}d_0'')$ for any weak injective $R$-module $W$ and $i\geq
1$ by Lemma $\ref{lemma1}$. It is easy to verify that  $\mbox{Ext}^i_R(W,\mbox{Ker}d_0)=0$  for any weak injective $R$-module $W$ and $i\geq
1$. In particular, the sequence $\xymatrix@C=0.5cm{
  0 \ar[r] & \mbox{Ker}d_0 \ar[r]^{} & I_0'\oplus I_0'' \ar[r]^{} & M \ar[r] & 0 }$ is $\mbox{Hom}_R(\mathcal{WI},-)$-exact.
Repeating this process, we may get the following commutative diagram:
$$
 \xymatrix{
 &\vdots\ar[d]&\vdots\ar[d]&\vdots\ar[d]&\\
 0\ar[r]&I_1'\ar[r]^{\!\!{1\choose 0}}\ar[d]^{d_1'}&I_1'\oplus I_1''\ar[r]^{\ (0\ 1)}\ar[d]^{d_1}&I_1''\ar[r]\ar[d]^{d_1''}&0\\
 0\ar[r]&I_0'\ar[r]^{\!\!{1\choose 0}}\ar[d]^{d_0'}&I_0'\oplus I_0''\ar[r]^{\ (0\ 1)}\ar[d]^{d_0}&I_0''\ar[r]\ar[d]^{d_0''}&0\\
 0\ar[r]&L\ar[r]^{f}\ar[d]&M\ar[r]^{g}\ar[d]&N\ar[r]\ar[d]&0\\
 &0&0&0&
 }
 $$
and a $\mbox{Hom}_R(\mathcal{WI},-)$-exact exact sequence $\xymatrix@C=0.5cm{
  \cdots \ar[r] & I_1'\oplus I_1'' \ar[r]^{} &I_0'\oplus I_0'' \ar[r]^{} & M \ar[r]^{}  & 0}$, where each $I_i'\oplus I_i''$ is injective. Therefore, $M$ is Gorenstein weak injective.
\end{proof}

Let $\mathcal{C}$ be a class of $R$-modules. Recall from \cite{Ho} that $\mathcal{C}$ is \emph{injectively resolving} if the class $\mathcal{I}$ of injective $R$-modules satisfies $\mathcal{I}\subseteq \mathcal{C}$, and for any exact sequence $0\rightarrow L\rightarrow M\rightarrow N\rightarrow 0$ with $L\in \mathcal{C}$, $M\in \mathcal{C}$ if and only if $N\in \mathcal{C}$.

We have the following proposition.

\begin{proposition}\label{injectively coresolving}
The class $\mathcal{GWI}$ is injectively resolving.
\end{proposition}

\begin{proof}
It is obvious that $\mathcal{I}\subseteq \mathcal{GWI}$. So, for any exact sequence $0\rightarrow L\rightarrow M\rightarrow N\rightarrow 0$ with $L$ Gorenstein weak injective, it suffices to show that if $M$ is Gorenstein weak injective, then so is $N$ by Proposition $\ref{extension closed}$. It is easy to verify that $\mbox{Ext}^i_R(W,N)=0$ for any weak injective $R$-module $W$ and any $i\geq
1$. Since $M$ is Gorenstein weak injective, we have an exact sequence $0\rightarrow G\rightarrow I_0\rightarrow M\rightarrow 0$ with $I_0$  injective and $G$ Gorenstein weak injective by Corollary $\ref{coro2.9}$. Consider the following pull-back diagram:
 $$
 \xymatrix{
 &0\ar[d]&0\ar[d]&&\\
 &G\ar@{=}[r]\ar[d]&G\ar[d]&&\\
 0\ar[r]&Q\ar[r]\ar[d]&I_0\ar[r]\ar[d]&N\ar[r]\ar@{=}[d]&0\\
 0\ar[r]&L\ar[r]\ar[d]&M\ar[r]\ar[d]&N\ar[r]&0\\
 &0&0&&
 }
 $$
 By Proposition $\ref{extension closed}$ and the second column in the above diagram, we have that $Q$ is Gorenstein weak injective.
 Thus $\mbox{Ext}^1_R(W,Q)=0$ for any weak injective $R$-module $W$, and so the exact sequence $(\natural):
  0 \rightarrow Q \rightarrow I_0 \rightarrow N \rightarrow 0 $ is $\mbox{Hom}_R(\mathcal{WI},-)$-exact. In addition, since $Q$ is Gorenstein weak injective, by the definition of Gorenstein weak injective $R$-modules, we have a $\mbox{Hom}_R(\mathcal{WI},-)$-exact exact sequence $\xymatrix@C=0.5cm{
  \cdots \ar[r] & I_2 \ar[r]^{} &I_1 \ar[r]^{} & Q \ar[r]^{}  & 0}$, where each $I_i$ is injective. Assembling this sequence with the sequence $(\natural)$, we get the following commutative diagram
$$\xymatrix@=0.4cm{
  \cdots \ar[rr] && I_2 \ar[rr]^{} &&I_1 \ar[rd] \ar[rr]^{} && I_0\ar[rr]&&N \ar[rr]^{}  && 0\\
  &&&&&Q\ar[ru]\ar[rd]&&&&&\\
  &&&&0\ar[ru]&&0&&&&}$$
which shows that $N$ is Gorenstein weak injective.
\end{proof}

\begin{proposition}\label{direct summands}
The class $\mathcal{GWI}$ is closed under direct summands.
\end{proposition}

\begin{proof}
It follows from Remark $\ref{1-4}$(3), Proposition $\ref{injectively coresolving}$ and \cite[Prop. 1.4]{Ho}.
\end{proof}

\begin{definition}\label{Gorenstein weak dimension}
The \emph{Gorenstein weak injective dimension} of an $R$-module $M$, denoted by $Gwid_R(M)$, is defined as $\operatorname{inf}\{n\mid \mbox{there is an exact sequence }0\rightarrow M\rightarrow G^0\rightarrow G^1\rightarrow \cdots\rightarrow G^n\rightarrow 0 \mbox{ with }  G^i \mbox{ Gorenstein weak injective for any } 0\leq i\leq n\}.$ If no such $n$ exists, set $Gwid_R(M)=\infty$.
\end{definition}

\begin{lemma}\label{gather}
Let $M$ be an $R$-module. Consider the following exact sequences:
\begin{gather*}
  0\rightarrow M\rightarrow G^0 \rightarrow G^1 \rightarrow\cdots \rightarrow G^{n-1}\rightarrow V^n\rightarrow 0\\
  0\rightarrow M\rightarrow \overline{G}^0 \rightarrow \overline{G}^1 \rightarrow\cdots \rightarrow  \overline{G}^{n-1}\rightarrow \overline{V}^n\rightarrow 0
\end{gather*}
where all $G^i$ and $\overline{G}^i$ are Gorenstein weak injective. Then $V^n$ is Gorenstein weak injective if and only if  $\overline{V}^n$ is Gorenstein weak injective.
\end{lemma}

\begin{proof}
It follows from the dual version of \cite[Lem. 3.12]{AB}.
\end{proof}

The following proposition shows the existence of Gorenstein weak injective preenvelope of modules.

\begin{proposition}\label{Gorenstein weak injective preenvelope}
Let $M$ be an $R$-module with finite Gorenstein weak injective dimension $n$. Then $M$ admits an injective Gorenstein weak injective preenvelope $\phi: M\hookrightarrow G$, where $V=\operatorname{Coker}\phi$ satisfies $id_R(V)=n-1$ (if $n=0$, this should be interpreted as $V=0$). Moreover, if $wid_R(M)<\infty$, then $G$ is injective.
\end{proposition}

\begin{proof}
Assume that $Gwid_R(M)=n$. Choose a partial injective resolution of $M$:$$\xymatrix@C=0.5cm{
  0 \ar[r] & M \ar[r]^{} & I^0 \ar[r]^{} & I^1 \ar[r]^{} & \cdots \ar[r]^{} & I^{n-1} \ar[r]^{} & G' \ar[r] & 0 }.$$ By Lemma $\ref{gather}$, we have that $G'$ is Gorenstein weak injective, and hence, by the definition of Gorenstein weak injective $R$-modules, there is an exact sequence
  $$\emph{{{\bf G}}}: \ \xymatrix@C=0.5cm{
  0 \ar[r] & \widehat{G} \ar[r]^{} & Q^0 \ar[r]^{} & Q^1 \ar[r]^{} & \cdots \ar[r]^{} & Q^{n-1} \ar[r]^{} & G' \ar[r] & 0 },$$
  where $\widehat{G}$ is Gorenstein weak injective and each $Q^i$ is injective, such that $\mbox{Hom}_R(W,{{\bf G}})$ is exact for any weak injective $R$-module $W$. In particular, $\mbox{Hom}_R(E,{{\bf G}})$ is exact for any  injective $R$-module $E$, and hence, by \cite[Sec. 8.1]{EJ}, there exists morphisms $I^i\rightarrow Q^i$ and $M\rightarrow \widehat{G}$ such that the following diagram is commutative£º
$$\xymatrix@C=0.5cm{
  0 \ar[r] & M \ar[r]^{}\ar[d] & I^0 \ar[r]^{}\ar[d] & I^1 \ar[r]^{}\ar[d] & \cdots \ar[r]^{} & I^{n-1} \ar[r]^{}\ar[d] &G' \ar[r] \ar@{=}[d]& 0 \\
0 \ar[r] & \widehat{G} \ar[r]^{} & Q^0 \ar[r]^{} & Q^1 \ar[r]^{} & \cdots \ar[r]^{} & Q^{n-1} \ar[r]^{} & G' \ar[r] & 0 }$$
This diagram gives a chain map of complexes as follows:
$$\xymatrix@C=0.5cm{
  0 \ar[r] & M \ar[r]^{}\ar[d] & I^0 \ar[r]^{}\ar[d] & I^1 \ar[r]^{}\ar[d] & \cdots \ar[r]^{} & I^{n-1} \ar[r]^{}\ar[d] &  0 \\
0 \ar[r] & \widehat{G} \ar[r]^{} & Q^0 \ar[r]^{} & Q^1 \ar[r]^{} & \cdots \ar[r]^{} & Q^{n-1} \ar[r]^{} &0 }$$
which induces an isomorphism in homology. Its mapping cone
$$\xymatrix@C=0.5cm{
  0 \ar[r] & M \ar[r]^{} & I^0\oplus \widehat{G} \ar[r]^{} & I^1\oplus Q^0 \ar[r]^{} & \cdots \ar[r]^{} & I^{n-1}\oplus Q^{n-2} \ar[r]^{} & Q^{n-1} \ar[r] & 0 }$$
is exact. It is obvious that all $I^i\oplus Q^{i-1}$ and $Q^{n-1}$ are injective, and $I^0\oplus \widehat{G}$ is Gorenstein weak injective. Let $V=\mbox{Coker}(M\rightarrow I^0\oplus \widehat{G})$. Then $id_R(V)\leq n-1$ (In fact, $id_R(V)= n-1$. Otherwise, $Gwid_R(M)<n$, which is a contradiction).

Since $id_R(V)<\infty$, it is easy to verify that $\mbox{Ext}^1_R(V,\overline{G})=0$ for any Gorenstein weak injective $R$-module $\overline{G}$. Thus the sequence $\mbox{Hom}_R(I^0\oplus \widehat{G},\overline{G})\rightarrow \mbox{Hom}_R(M,\overline{G})\rightarrow 0$ is exact. Let $G=I^0\oplus \widehat{G}$. Then $M\rightarrow G$ is a Gorenstein weak injective preenvelope of $M$, as desired.

Since $wid_R(M)<\infty$, $wid_R(G)<\infty$. By Proposition $\ref{either or}$, we have that $G$ is injective.
\end{proof}

\begin{corollary}
Given an exact sequence $0\rightarrow L\rightarrow M\rightarrow N\rightarrow 0$. If $M$ and $N$ are Gorenstein weak injective, then the following are equivalent:

$(1)$ $L$ is Gorenstein weak injective;

$(2)$  $L$ is  Ding injective;

$(3)$ $L$ is Gorenstein  injective;

$(4)$ $\operatorname{Ext}^1_R(I,L)=0$ for any  injective $R$-module $I$;

$(5)$ $\operatorname{Ext}^1_R(E,L)=0$ for any  FP-injective $R$-module $E$;

$(6)$ $\operatorname{Ext}^1_R(W,L)=0$ for any weak injective $R$-module $W$.
\end{corollary}

\begin{proof} (1) $\Rightarrow$ (6) $\Rightarrow$ (5) $\Rightarrow$ (4) are trivial. (2) $\Leftrightarrow$ (5) and (3) $\Leftrightarrow$ (4) follow from the dual versions of \cite[Cor. 2.11]{Ho} and \cite[Cor. 2.1]{MT}.

(4) $\Rightarrow$ (1) By hypothesis, $Gwid_R(L)\leq 1$, and so there is an exact sequence $0\rightarrow L\rightarrow G\rightarrow I\rightarrow 0$ with $G$ Gorenstein weak injective and $I$ injective. By assumption, $\mbox{Ext}^1_R(I,L)=0$, and thus this sequence is split. Hence $L$ is Gorenstein weak injective by Proposition $\ref{direct summands}$.
\end{proof}

Now we give a functorial description of Gorenstein weak injective dimension of modules.

\begin{proposition}\label{Gweakdim} Let $M$ be an $R$-module with  finite Gorenstein weak injective dimension. Then the following are equivalent:

$(1)$ $Gwid_R(M)\leq n$;

$(2)$ $\operatorname{Ext}^i_R(W,M)=0$ for any weak injective $R$-module $W$ and any $i\geq n+1$;

$(3)$ $\operatorname{Ext}^i_R(\widetilde{W},M)=0$ for any  $R$-module $\widetilde{W}$ with finite weak injective dimension and any $i\geq n+1$;

$(4)$ For every exact sequence $\xymatrix@C=0.5cm{
  0 \ar[r] & M \ar[r]^{} & I^0 \ar[r]^{} & \cdots \ar[r]^{} & I^{n-1} \ar[r]^{} & C^n \ar[r] & 0
  }$, where each $I^i$ is  injective, $C^n$ is Gorenstein weak injective.

$(5)$ For every exact sequence $\xymatrix@C=0.5cm{
  0 \ar[r] & M \ar[r]^{} & G^0 \ar[r]^{} & \cdots \ar[r]^{} & G^{n-1} \ar[r]^{} & V^n \ar[r] & 0
  }$, where each $G^i$ is Gorenstein weak injective, $V^n$ is Gorenstein weak injective.

  Consequently, the Gorenstein weak injective dimension of $M$ is determined by the formulas:

$$
  \begin{array}{cc}
  Gwid_R(M)&\!\!\!\!\!\!\!\!\!\!\!\!\!\!\!\!\!\!\!\!=\operatorname{sup}\{n\mid \operatorname{Ext}^i_R(W,M)\neq 0 \mbox{ for some weak injective $R$-module }W\}\\
  &\!\!=\operatorname{sup}\{n\mid \operatorname{Ext}^i_R(\widetilde{W},M)\neq 0 \mbox{ for some  $R$-module $\widetilde{W}$ with  }wid_R(\widetilde{W})<\infty\}.
  \end{array}
$$
\end{proposition}

\begin{proof} (1)$\Rightarrow$(2). Since $Gwid_R(M)\leq n$, we have an exact sequence
$$
\xymatrix@C=0.5cm{
  0 \ar[r] & M \ar[r]^{} & G^0 \ar[r]^{} & G^1 \ar[r]^{} & \cdots \ar[r]^{} & G^n \ar[r] & 0 },
  $$
  where each $G^i$ is Gorenstein weak injective. Let $V^1=\mbox{Coker}(M\rightarrow
  G^0)$, $V^i=\mbox{Coker}(G^{i-2}\rightarrow
  G^{i-1})$, $2\leq i\leq n$. Then for any weak injective $R$-module $W$, we have
  $$
  \mbox{Ext}_R^i(W,M)\cong\mbox{Ext}_R^{i-1}(W,V^1)\cong\cdots\cong\mbox{Ext}_R^{i-n}(W,G^n)=0, \ i\geq n+1.$$

(2)$\Rightarrow$(3) holds by dimension shifting.

  (3)$\Rightarrow$(5).  For every exact sequence $\xymatrix@C=0.4cm{
  0 \ar[r] & M \ar[r]^{} & G^0 \ar[r]^{} & \cdots \ar[r]^{} & G^{n-1} \ar[r]^{} & V^n \ar[r] & 0
  }$, where each $G^i$ is Gorenstein weak injective. Let $V^0=M$,, $V^1=\mbox{Coker}(M\rightarrow
  G^0)$ and $V^j=\mbox{Coker}(G^{j-2}\rightarrow G^{j-1})$, $2\leq j\leq n$. Then  the sequence $
    0 \rightarrow V^j \rightarrow E^j \rightarrow V^{j+1} \rightarrow 0 $ is exact for any $0\leq j\leq n-1$. Let $W$ be any weak injective $R$-module. Then we have the following exact sequence
    $$
    \xymatrix@C=0.5cm{
      \mbox{Ext}_R^i(W,G^j) \ar[r] & \mbox{Ext}_R^i(W,V^{j+1}) \ar[r]^{} & \mbox{Ext}_R^{i+1}(W,V^j) \ar[r] &
      \mbox{Ext}_R^{i+1}(W,G^j)
      },
      $$
      where $\mbox{Ext}_R^i(W,G^j)=0=\mbox{Ext}_R^{i+1}(W,G^j)$ by Lemma $\ref{lemma1}$. So we have
$$\mbox{Ext}_R^{i}(W,V^n)\cong\mbox{Ext}_R^{i+1}(W,V^{n-1})\cong\cdots\cong\mbox{Ext}_R^{i+n}(W,M)=0, \ i\geq 1.$$
Moreover, since $Gwid_R(M)<\infty$,  $Gwid_R(V^n)<\infty$, and hence we have the following exact sequence for some non-negative integer $m$
$$
\xymatrix@C=0.5cm{
  0 \ar[r] & V^n \ar[r]^{} & {\overline{G}}^0 \ar[r]^{} & {\overline{G}}^1 \ar[r]^{} & \cdots \ar[r]^{} & {\overline{G}}^m \ar[r] & 0 },
  $$
 where each ${\overline{G}}^i$ is Gorenstein weak injective. Let ${\overline{V}}^0=V^n$, ${\overline{V}}^1=\mbox{Coker}(V^n\rightarrow {\overline{G}}^0)$
and ${\overline{V}}^i=\mbox{Coker}({\overline{G}}^{i-2}\rightarrow {\overline{G}}^{i-1})$, $2\leq
i\leq m$. Then we have
$$\mbox{Ext}_R^1(W,{\overline{V}}^{m-1})\cong\mbox{Ext}_R^2(W,{\overline{V}}^{m-2})\cong\cdots\cong\mbox{Ext}_R^m(W,V^n)=0.$$
Since ${\overline{V}}^{m-1}=\mbox{Coker}({\overline{G}}^{m-3}\rightarrow {\overline{G}}^{m-2})$, we have an exact sequence $0\rightarrow {\overline{V}}^{m-1}\rightarrow {\overline{G}}^{m-1}\rightarrow {\overline{G}}^{m}\rightarrow 0$, that is, $Gwid_R({\overline{V}}^{m-1})\leq 1$. By Proposition $\ref{Gorenstein weak injective preenvelope}$, there exists an exact sequence $0\rightarrow {\overline{V}}^{m-1}\rightarrow G\rightarrow I\rightarrow 0$ such that $G$ is Gorenstein weak injective and $I$ is injective. In addition, this sequence is split since $\mbox{Ext}_R^1(I,{\overline{V}}^{m-1})=0$. Therefore, ${\overline{V}}^{m-1}$ is Gorenstein weak injective. By a similar argument, we have that ${\overline{V}}^{m-2}$, $\cdots$, ${\overline{V}}^{0}$ are Gorenstein weak injective. In particular, $V^n$ is Gorenstein weak injective.

(5) $\Rightarrow$(4) and (4) $\Rightarrow$(1) are trivial.
\end{proof}

\begin{proposition}
Given an exact sequence $
    0 \rightarrow L \rightarrow M \rightarrow N \rightarrow 0 $. If any two of $R$-modules $L$, $M$, or $N$ have finite Gorenstein weak injective dimension, then so has the third.
\end{proposition}

It is natural to investigate how much the usual injective dimension differs from the Gorenstein weak injective one. In what follows, $Gid_R(M)$ and $Did_R(M)$ will denote respectively the Gorenstein injective and Ding injective dimension of an $R$-module $M$ (see \cite[Def. 2.8]{Ho} and \cite[Def. 2.3]{Ya} for details).

\begin{proposition}
Let $M$ be an $R$-module. Then

$(1)$ $Gwid_R(M)\leq id_R(M)$ with equality, if $wid_R(M)<\infty$;

$(2)$ $Gid_R(M)\leq Did_R(M)\leq Gwid_R(M)$ with equalities, if $Gwid_R(M)<\infty$.
\end{proposition}

\begin{proof}
(1) Clearly, $Gwid_R(M)\leq id_R(M)$. Let $wid_R(M)<\infty$. It suffices to prove $id_R(M)\leq Gwid_R(M)$. Without loss of generality, we assume that $Gwid_R(M)=n<\infty$ for some non-negative integer $n$. If $n=0$, that is, $M$ is Gorenstein weak injective, then there is an exact sequence $
    0 \rightarrow L \rightarrow I \rightarrow M \rightarrow 0 $ with $I$  injective and $L$ Gorenstein weak injective. Note that $\mbox{Ext}^1_R(M,L)=0$ since $wid_R(M)<\infty$. Thus this sequence is split, and so $M$ is injective. Now let $n\geq 1$. By Proposition $\ref{Gorenstein weak injective preenvelope}$, there is an exact sequence $
    0 \rightarrow M \rightarrow G \rightarrow V \rightarrow 0 $ with $G$ Gorenstein weak injective and $id_R(V)=n-1$. Since $G$ is Gorenstein weak injective, there is an exact sequence $
    0 \rightarrow G' \rightarrow I \rightarrow G \rightarrow 0 $ with $I$  injective and $G'$ Gorenstein weak injective. Consider the following pull-back diagram:
    \begin{gather*}
\xymatrix{
&0\ar[d]&0\ar[d]&&\\
&G'\ar@{=}[r]\ar[d]&G'\ar[d]&&\\
0\ar[r]&N\ar[r]\ar[d]&I\ar[r]\ar[d]&V\ar@{=}[d]\ar[r]&0\\
0\ar[r]&M\ar[r]\ar[d]&G\ar[r]\ar[d]&V\ar[r]&0\\
&0&0&&
}
\end{gather*}
Since $I$ is injective and $id_R(V)=n-1$, it follows from the middle row in the above diagram that $id_R(N)\leq n$. Moreover, $\mbox{Ext}^1_R(M,G')=0$ since $wid_R(M)<\infty$, which shows that the second column in the above diagram is split. Thus $id_R(M)\leq id_R(N)\leq n=Gwid_R(M)$, as desired.

(2) Since every Gorenstein weak injective $R$-module is Ding injective, and every Ding injective $R$-module is Gorenstein injective, it is obvious that $Gid_R(M)\leq Did_R(M)\leq Gwid_R(M)$ for any $R$-module $M$. Now let $Gwid_R(M)=n<\infty$. In order to show $Gid_R(M)\geq n$, it suffices to find an injective $R$-module $I$ such that $\mbox{Ext}^n_R(I,M)\neq 0$ by \cite[Thm. 2.22]{Ho}. Since $Gwid_R(M)=n$, there is some weak injective $R$-module $W$ such that $\mbox{Ext}^n_R(W,M)\neq 0$. Consider an exact sequence $
    0 \rightarrow W \rightarrow I \rightarrow W' \rightarrow 0 $ with $I$ injective. It is easy to verify that $W'$ is weak injective. Consider the following exact sequence
    $$
    \cdots\rightarrow \mbox{Ext}^n_R(I,M)\rightarrow \mbox{Ext}^n_R(W,M)\rightarrow \mbox{Ext}^{n+1}_R(W',M)=0.
    $$
 It follows then that   $\mbox{Ext}^n_R(I,M)\neq 0$, as desired.
\end{proof}

\begin{corollary}
If an $R$-module $M$ is Gorenstein injective or Ding injective, then  either $M$ is Gorenstein weak injective or $Gwid_R(M)=\infty$.
\end{corollary}

Accordingly, we define the left global Gorenstein weak injective dimension, $\ell. GwiD(R)$, of a ring $R$ as follows:
$$\ell. GwiD(R)=\mbox{sup}\{Gwid_R(M)\mid M \mbox{ is any $R$-module}\}.$$

\begin{corollary}
If $\ell. GwiD(R)<\infty$, then the following are equivalent:

$(1)$ $\ell. GwiD(R)\leq n$;

$(2)$ $pd_R(\widetilde{W})\leq n$ for any $R$-module $\widetilde{W}$ with finite weak injective dimension;

$(3)$ $pd_R({W})\leq n$ for any weak injective $R$-module $W$.
\end{corollary}

\begin{proof}
(1) $\Rightarrow$ (2). Let $\widetilde{W}$ be an $R$-module  with finite weak injective dimension. For any $R$-module $M$, $Gwid_R(M)\leq n$ by hypothesis, and hence  we have $\mbox{Ext}^i_R(\widetilde{W},M)=0$ for any $i\geq n+1$. So $pd_R(\widetilde{W})\leq n$.

(2) $\Rightarrow$ (3) is trivial.

(3) $\Rightarrow$ (1). Let $M$ be any $R$-module. It follows from Proposition $\ref{Gweakdim}$ that $Gwid_R(M)\leq n$. Thus  $\ell. GwiD(R)\leq n$.
\end{proof}

It is well-known that pure injective  modules play an important role in homological algebra, and  the relative version also have been investigated by many authors (e.g. \cite{BGMS,Fa,Hu1,Si,Wa,Wi}). Inspired by this, we give the following definition.

\begin{definition}
An exact sequence $0\rightarrow L\rightarrow M\rightarrow N\rightarrow 0$  is called \emph{$\mathcal{GWI}$-copure exact} if for any $X\in \mathcal{GWI}$, the induced sequence $0\rightarrow \mbox{Hom}_R(N,X)\rightarrow \mbox{Hom}_R(M,X)\rightarrow \mbox{Hom}_R(L,X)\rightarrow 0$ is exact.

An $R$-module $M$ is called \emph{$\mathcal{GWI}$-copure projective} (resp. \emph{$\mathcal{GWI}$-copure injective}) if $\mbox{Hom}_R(M,-)$ (resp. $\mbox{Hom}_R(-,M)$) leaves any $\mathcal{GWI}$-copure  exact sequence exact.
\end{definition}

If $\xymatrix@C=0.5cm{
  0\ar[r]& L\ar[r]^{f} & M \ar[r]^{g} & N \ar[r]^{} & 0  }$   is a $\mathcal{GWI}$-copure exact sequence, then $f$ is called a \emph{$\mathcal{GWI}$-copure injection}, and $g$ is called a \emph{$\mathcal{GWI}$-copure surjection}.

The following proposition shows the necessity of studying $\mathcal{GWI}$-copure exact sequences.

\begin{proposition}\label{prop5.2}
 Let $M$ be an $R$-module with finite Gorenstein weak injective dimension. Then the following are equivalent:

$(1)$ $M$ is  Gorenstein weak injective;

$(2)$ For any $\mathcal{GWI}$-copure injection $i:X\rightarrow Y$  and any $h:X\rightarrow M$, there exists $g:Y\rightarrow M$ such that the following diagram commute:
$$\xymatrix{
 X\ar[d]_{h} \ar[r]^{i} &  Y  \ar@{.>}[dl]^{g}    \\
  M                      }
$$

$(3)$ The functor $\operatorname{Hom}_R(-,M)$ is exact with respect to any $\mathcal{GWI}$-copure exact sequence;

$(4)$ Every $\mathcal{GWI}$-copure exact sequence $\xymatrix@C=0.5cm{
 0\ar[r]& M \ar[r]^{f} & N \ar[r]^{g} & L \ar[r]^{} & 0  }$ is split.
\end{proposition}

\begin{proof}
(1) $\Rightarrow$ (2). Since $i:X\rightarrow Y$ is $\mathcal{GWI}$-copure injective, we have a $\mathcal{GWI}$-copure exact sequence $\xymatrix@C=0.5cm{
 0\ar[r]& X \ar[r]^{i} & Y \ar[r]^{j} & Z\ar[r]^{} & 0  }$. Applying the functor $\mbox{Hom}_R(-,M)$ to it, since  $M$ is Gorenstein weak injective, we have the following exact sequence
  $$
  \xymatrix@C=0.5cm{
    0 \ar[r] & \mbox{Hom}_R(Z,M) \ar[r]^{j^*} & \mbox{Hom}_R(Y,M) \ar[r]^{i^*} & \mbox{Hom}(X,M) \ar[r] & 0 }.
  $$
  Thus for any  $h:X\rightarrow M$, there exists $g:Y\rightarrow M$ such that $i^*(g)=gi=h$, as desired.

(2) $\Rightarrow$ (3).  Given a $\mathcal{GWI}$-copure exact sequence $\xymatrix@C=0.5cm{
 0\ar[r]& X \ar[r]^{i} & Y \ar[r]^{j} & Z\ar[r]^{} & 0  }$. Since the functor $\mbox{Hom}(-,M)$ is left exact, it suffices to show that $i^*$ is surjective. For a map $h:X\rightarrow M$, by assumption, there exists $g:Y\rightarrow M$ such that $gi=h$. Thus, if $h\in \mbox{Hom}_R(X,M)$, then $h=gi=i^*(g)\in \mbox{im}i^*$, and so $i^*$ is surjective. Hence the functor $\mbox{Hom}_R(-,M)$ is exact with respect to any $\mathcal{GWI}$-copure exact sequence.

(3) $\Rightarrow$ (4).   Since $\xymatrix@C=0.5cm{
0\ar[r]&  M \ar[r]^{f} & N \ar[r]^{g} & L \ar[r]^{} & 0  }$ is a $\mathcal{GWI}$-copure exact sequence, we have the following exact sequence
 $$
  \xymatrix@C=0.5cm{
    0 \ar[r] & \mbox{Hom}_R(L,M) \ar[r]^{g^*} & \mbox{Hom}_R(N,M) \ar[r]^{f^*} & \mbox{Hom}_R(M,M) \ar[r] & 0 },
  $$
 that is, there exists a map $f':N\rightarrow M$ such that $f^*(f')=f'f=\mbox{Id}_M$.  It follows then that the sequence $\xymatrix@C=0.5cm{
 0\ar[r]& M \ar[r]^{f} & N \ar[r]^{g} & L \ar[r]^{} & 0 }$ is split.

  (4)$\Rightarrow$(1). Since $M$ has finite Gorenstein weak injective dimension, by Proposition $\ref{Gweakdim}$, there exists an exact sequence  $\xymatrix@C=0.5cm{
0\ar[r]&  M \ar[r]^{f} & G \ar[r]^{g} & V \ar[r]^{} & 0  }$  such that $f$ is a Gorenstein weak injective preenvelope and $id_R(V)= Gwid_R(M)-1$.  Thus this sequence is $\mathcal{GWI}$-copure exact sequence. By (4), it is split, that is, $G\cong M\oplus V$. Hence $M$ is Gorenstein weak injective.\end{proof}

\begin{proposition}
 $$\mathcal{GWI}=\{\mathcal{GWI}\mbox{-copure injective $R$-modules} \}\bigcap \widetilde{\mathcal{GWI}},$$ where $\widetilde{\mathcal{GWI}}$ denotes the class of $R$-modules with finite Gorenstein weak injective dimension.
\end{proposition}

\begin{proof}
 Clearly, every Gorenstein weak injective $R$-module is $\mathcal{GWI}$-copure injective. Let $M$ be  $\mathcal{GWI}$-copure injective  and have finite Gorenstein weak injective dimension. Then  there exists an exact sequence  $\xymatrix@C=0.5cm{
0\ar[r]&  M \ar[r]^{f} & G \ar[r]^{g} & V \ar[r]^{} & 0  }$  such that $f$ is a Gorenstein weak injective preenvelope and $id_R(V)= Gwid_R(M)-1<\infty$.  It follows then that $\mbox{Ext}^1_{R}(V,G')=0$ for any Gorenstein weak injective $R$-module $G'$. So this sequence is in fact $\mbox{Hom}_R(-,\mathcal{GWI})$-exact, that is, it is  $\mathcal{GWI}$-copure exact. Moreover, since $M$ is  $\mathcal{GWI}$-copure injective, we have the following exact sequence
$$
\xymatrix@C=0.5cm{
  0 \ar[r] & \mbox{Hom}_R(V,M) \ar[r]^{g^*} & \mbox{Hom}_R(G,M) \ar[r]^{f^*} & \mbox{Hom}_R(M,M) \ar[r] & 0 }.
$$
Thus, $\xymatrix@C=0.5cm{
 0\ar[r]& M \ar[r]^{f} & G \ar[r]^{g} & V \ar[r]^{} & 0  }$ is split, and hence $M$ is Gorenstein weak injective as a direct summand of $G$.
\end{proof}

\section{Gorenstein weak projective modules and dimension}

In this section, we give the definition of Gorenstein weak projective modules in terms of weak flat modules,  and discuss some of the properties of these modules. The results and their proofs in this section are completely
dual to that in Section 2, so we only list the results without proofs.

\begin{definition}\label{Gorenstein weak projective}
An $R$-module $M$ is called \emph{Gorenstein weak projective} if there exists an exact sequence of  projective $R$-modules
$$
 \emph{\textbf{P}}=\xymatrix@C=0.5cm{
  \cdots \ar[r] & P_1 \ar[r]^{} & P_0 \ar[r]^{} & P^0 \ar[r]^{} & P^1 \ar[r] & \cdots }
$$
such that $M=\operatorname{Coker}(P_1\rightarrow P_0)$ and the functor $\operatorname{Hom}_R(-,W)$ leaves this sequence exact whenever $W$ is a weak flat $R$-module.
\end{definition}

\begin{remark}

$(1)$ Every projective $R$-module is Gorenstein weak projective.

$(2)$ Since every flat $R$-module is weak flat, every Gorenstein weak projective is Ding projective (in the sense of \cite{Gi}). If $R$ is a left coherent ring, then the class of  Gorenstein weak projective $R$-modules coincides with the class of  Ding projective $R$-modules. Moreover, we have the following implications:
$$
\begin{array}{ccc}
\mbox{Gorenstein weak projective }R\mbox{-modules}&\Rightarrow&\mbox{Ding projective }R\mbox{-modules}\\
&\Rightarrow&\mbox{Gorenstein projective }R\mbox{-modules}.
\end{array}
$$
If $R$ is an $n$-Gorenstein ring (i.e. a left and right Noetherian ring with self-injective dimension at most
$n$ on both sides for some non-negative integer $n$), then these three kinds of $R$-modules coincide.

$(3)$ The class of Gorenstein weak projective $R$-modules is closed under direct sums.

$(4)$ If $
 \emph{\textbf{P}}=\xymatrix@C=0.5cm{
  \cdots \ar[r] & P_1 \ar[r]^{} & P_0 \ar[r]^{} & P^0 \ar[r]^{} & P^1 \ar[r] & \cdots }
$ is an exact sequence of  projective $R$-modules such that the functor $\operatorname{Hom}_R(-,W)$ leaves this sequence exact whenever $W$ is a weak flat $R$-module, then by symmetry, all the images, the kernels and the cokernels of $\emph{\textbf{P}}$ are Gorenstein weak flat.
\end{remark}

\begin{proposition}
A Gorenstein weak projective $R$-module is either projective or has  weak flat dimension $\infty$. Consequently, $\mathcal{GWP}\bigcap \widetilde{\mathcal{WF}}=\mathcal{P}$, where $\widetilde{\mathcal{WF}}$ denote the class of $R$-modules with finite weak flat dimension.
\end{proposition}

\begin{corollary}
An $R$-module is projective if and only if it is weak flat and Gorenstein weak projective.
\end{corollary}

Let $\mathcal{C}$ be a class of $R$-modules. Recall from \cite{Ho} that $\mathcal{C}$ is \emph{projectively resolving} if the class $\mathcal{P}$ of projective $R$-modules satisfies $\mathcal{P}\subseteq \mathcal{C}$, and for any exact sequence $0\rightarrow L\rightarrow M\rightarrow N\rightarrow 0$ with $N\in \mathcal{C}$, $L\in \mathcal{C}$ if and only if $M\in \mathcal{C}$.

We have the following proposition.

\begin{proposition}
The class $\mathcal{GWP}$ is projectively resolving and closed under direct summands.
\end{proposition}

\begin{definition}
The \emph{Gorenstein weak projective dimension} of an $R$-module $M$, denoted by $Gwpd_R(M)$, is defined as $\operatorname{inf}\{n\mid \mbox{there is an exact sequence }0\rightarrow G_n\rightarrow \cdots\rightarrow G_1\rightarrow G_0\rightarrow M\rightarrow 0 \mbox{ with }  G_i \mbox{ Gorenstein weak projective for any } 0\leq i\leq n\}.$ If no such $n$ exists, set $Gwpd_R(M)=\infty$.
\end{definition}

\begin{proposition}
Let $M$ be an $R$-module with finite Gorenstein weak projective dimension $n$. Then $M$ admits a surjective Gorenstein weak projective precover $\varphi: G\hookrightarrow M$, where $K=\mbox{Ker}\varphi$ satisfies $pd_R(K)=n-1$ (if $n=0$, this should be interpreted as $K=0$). Moreover, if $wfd_R(M)<\infty$, then $G$ is projective.
\end{proposition}

\begin{corollary}
Given an exact sequence $0\rightarrow L\rightarrow M\rightarrow N\rightarrow 0$. If $L$ and $M$ are Gorenstein weak injective, then the following are equivalent:

$(1)$ $N$ is Gorenstein weak projective;

$(2)$  $N$ is  Ding projective (in the sense of \cite{Gi});

$(3)$ $N$ is Gorenstein  projective;

$(4)$ $\operatorname{Ext}^1_R(N,P)=0$ for any  projective $R$-module $P$;

$(5)$ $\operatorname{Ext}^1_R(N,F)=0$ for any  flat $R$-module $F$;

$(6)$ $\operatorname{Ext}^1_R(N,W)=0$ for any weak flat $R$-module $W$.
\end{corollary}

\begin{proposition} Let $M$ be an $R$-module with  finite Gorenstein weak projective dimension. Then the following are equivalent:

$(1)$ $Gwpd_R(M)\leq n$;

$(2)$ $\operatorname{Ext}^i_R(M,W)=0$ for any weak flat $R$-module $W$ and any $i\geq n+1$;

$(3)$ $\operatorname{Ext}^i_R(M,\widetilde{W})=0$ for any  $R$-module $\widetilde{W}$ with finite weak flat dimension and any $i\geq n+1$;

$(4)$ For every exact sequence $\xymatrix@C=0.5cm{
  0 \ar[r] & K_n \ar[r]^{} & P_{n-1} \ar[r]^{} & \cdots \ar[r]^{} & P_0 \ar[r]^{} & M \ar[r] & 0
  }$, where each $P_i$ is  projective, $K_n$ is Gorenstein weak projective.

$(5)$ For every exact sequence $\xymatrix@C=0.5cm{
  0 \ar[r] & K_n' \ar[r]^{} & G_{n-1} \ar[r]^{} & \cdots \ar[r]^{} & G_0 \ar[r]^{} & M \ar[r] & 0
  }$, where each $G_i$ is Gorenstein weak projective, $K_n'$ is Gorenstein weak projective.

  Consequently, the Gorenstein weak projective dimension of $M$ is determined by the formulas:
$$
  \begin{array}{cc}
  Gwpd_R(M)&\!\!\!\!\!\!\!\!\!\!\!\!\!\!\!\!\!\!\!\!\!\!\!\!\!\!\!\!\!\!\!\!\!=\operatorname{sup}\{n\mid \operatorname{Ext}^i_R(M, W)\neq 0 \mbox{ for some weak flat $R$-module }W\}\\
  &\!\!=\operatorname{sup}\{n\mid \operatorname{Ext}^i_R(M,\widetilde{W})\neq 0 \mbox{ for some  $R$-module $\widetilde{W}$ with  }wfd_R(\widetilde{W})<\infty\}.
  \end{array}
$$
\end{proposition}

In what follows, $Gpd_R(M)$ and $Dpd_R(M)$ will denote respectively the Gorenstein projective and Ding projective dimension of an $R$-module $M$ (see \cite[Def. 2.8]{Ho} and \cite[Def. 2.3]{Ya} for details).

\begin{proposition}
Let $M$ be an $R$-module. Then

$(1)$ $Gwpd_R(M)\leq pd_R(M)$ with equality, if $wfd_R(M)<\infty$;

$(2)$ $Gpd_R(M)\leq Dpd_R(M)\leq Gwpd_R(M)$ with equalities, if $Gwpd_R(M)<\infty$.
\end{proposition}

Accordingly, we define the left global Gorenstein weak projective dimension, $\ell. GwpD(R)$, of a ring $R$ as follows:
$$\ell. GwpD(R)=\mbox{sup}\{Gwpd_R(M)\mid M \mbox{ is any $R$-module}\}.$$

\begin{corollary}
If $\ell. GwpD(R)<\infty$, then the following are equivalent:

$(1)$ $\ell. GwpD(R)\leq n$;

$(2)$ $id_R(\widetilde{W})\leq n$ for any $R$-module $\widetilde{W}$ with finite weak flat dimension;

$(3)$ $id_R({W})\leq n$ for any weak flat $R$-module $W$.
\end{corollary}

\begin{definition}
 An exact sequence $0\rightarrow L\rightarrow M\rightarrow N\rightarrow 0$  is called \emph{$\mathcal{GWP}$-pure exact} if for any $X\in \mathcal{GWP}$, the induced sequence $0\rightarrow \mbox{Hom}_R(X,L)\rightarrow \mbox{Hom}_R(X,M)\rightarrow \mbox{Hom}_R(X,N)\rightarrow 0$ is exact.

An $R$-module $M$ is called \emph{$\mathcal{GWP}$-pure projective} (resp. \emph{$\mathcal{GWP}$-pure injective}) if $\mbox{Hom}_R(-,M)$ (resp. $\mbox{Hom}_R(M,-)$) leaves any $\mathcal{GWP}$-pure  exact sequence exact.
\end{definition}

If $\xymatrix@C=0.5cm{
  0\ar[r]& L\ar[r]^{f} & M \ar[r]^{g} & N \ar[r]^{} & 0  }$   is a $\mathcal{GWP}$-pure exact sequence, then $f$ is called a \emph{$\mathcal{GWP}$-pure injection}, and $g$ is called a \emph{$\mathcal{GWP}$-pure surjection}.

The following proposition shows the necessity of studying $\mathcal{GWP}$-pure exact sequences.

\begin{proposition}\label{prop5.2}
 Let $M$ be an $R$-module with finite Gorenstein weak projective dimension. Then the following are equivalent:

$(1)$ $M$ is  Gorenstein weak projective;

$(2)$ For any $\mathcal{GWP}$-pure surjection $j:X\rightarrow Y$  and any $g:M\rightarrow Y$, there exists $h:M\rightarrow X$ such that the following diagram commute:
$$\xymatrix{
  &M \ar@{.>}[ld]_{h}\ar[d]^{g} \\  X \ar[r]^{j} &  Y                   }
$$

$(3)$ The functor $\operatorname{Hom}_R(M,-)$ is exact with respect to any $\mathcal{GWP}$-pure exact sequence;

$(4)$ Every $\mathcal{GWP}$-pure exact sequence $\xymatrix@C=0.5cm{
 0\ar[r]& L \ar[r]^{f} & N \ar[r]^{g} & M\ar[r]^{} & 0  }$ is split.
\end{proposition}

\begin{proposition}\label{prop5.3}
 $$\mathcal{GWP}=\{\mathcal{GWP}\mbox{-pure projective $R$-modules} \}\bigcap \widetilde{\mathcal{GWP}},$$ where $\widetilde{\mathcal{GWP}}$ denotes the class of $R$-modules with finite Gorenstein weak projective dimension.
\end{proposition}

\section{Derived functors with respect to Gorenstein weak  modules}

In this section, we mainly investigate the homological properties of derived functors with respect to Gorenstein weak  modules (including Gorenstein weak injective and weak projective modules).

Following \cite[Def. 8.1.2]{EJ1} or \cite[1.5]{Ho}, we first give the following definition.

\begin{definition} Let $M$ be an $R$-module. A \emph{proper right
$\mathcal{GWI}$-resolution} of $M$ is a $\operatorname{Hom}_R(-,\mathcal{GWI})$-exact
complex
$$\xymatrix@C=0.5cm{
  0 \ar[r] & M \ar[r]^{} & G^0 \ar[r]^{} & G^1 \ar[r] & \cdots }$$
  with each $G^i$ Gorenstein weak injective.

Dually,  a \emph{proper left
$\mathcal{GWP}$-resolution} of $M$ is a $\operatorname{Hom}_R(\mathcal{GWP},-)$-exact
complex
$$\xymatrix@C=0.5cm{
  \cdots \ar[r] & G_1 \ar[r]^{} & G_0 \ar[r]^{} & M \ar[r] & 0 }$$
  with each $G_i$ Gorenstein weak projective.
   \end{definition}

Note that since injective $R$-module is Gorenstein weak injective,  every proper right
$\mathcal{GWI}$-resolution of an $R$-module is exact. Similarly, every   proper left
$\mathcal{GWP}$-resolution of an $R$-module is also exact.

Following Proposition $\ref{Gorenstein weak injective preenvelope}$, we have the following lemma which shows  the existence of proper right $\mathcal{GWI}$-resolutions (resp. proper left $\mathcal{GWP}$-resolutions) of an $R$-module.

\begin{lemma}
$(1)$ Assume that  $M$ is an $R$-module with finite Gorenstein weak injective dimension. Then $M$ admits a proper right $\mathcal{GWI}$-resolution. In particular, if $Gwid_R(M)=n<\infty$, then $M$ admits a proper right $\mathcal{GWI}$-resolution of length $n$.

$(2)$ Assume that  $M$ is an $R$-module with finite Gorenstein weak projective dimension. Then $M$ admits a proper left $\mathcal{GWP}$-resolution. In particular, if $Gwpd_R(M)=n<\infty$, then $M$ admits a proper left $\mathcal{GWP}$-resolution of length $n$.
\end{lemma}

We denote by    $\mbox{RightRes}_{_R\mathcal{M}}(\mathcal{GWI})$ and $\mbox{LeftRes}_{_R\mathcal{M}}(\mathcal{GWP})$ the full subcategory of $_R\mathcal{M}$ consisting of those $R$-modules that have a proper right $\mathcal{GWI}$-resolution and a proper left
$\mathcal{GWP}$-resolution, respectively.

Following \cite[2.4]{Ho2}, we define two kinds of right derived functors as follows: $$\mbox{Ext}^n_{\mathcal{GWI}}(M,-)=\mbox{R}^n_{\mathcal{GWI}}(M,-), \  \mbox{Ext}^n_{\mathcal{GWP}}(-,N)=\mbox{R}^n_{\mathcal{GWP}}(-,N)$$ for fixed $R$-modules $M$ and $N$, and wish to prove the following theorem:

\begin{theorem}\label{balanced functor}
Let $M, N$ be an $R$-modules with $Gwpd_R(M)<\infty$ and $Gwid_R(N)<\infty$. Then we have isomorphisms
$$\operatorname{Ext}^n_{\mathcal{GWI}}(M,N)\cong \operatorname{Ext}^n_{\mathcal{GWP}}(M,N), n\geq 0$$
which are functorial in $M$ and $N$.
\end{theorem}

Following \cite[Thm. 2.6]{Ho2}, it suffices to prove the following lemma.

\begin{lemma}\label{two exact}
$(1)$ Assume that  $N$ is an $R$-module with finite Gorenstein weak injective dimension. Let $\textbf{G}^+=0 \rightarrow N\rightarrow G^0\rightarrow G^1\rightarrow \cdots$ be a proper right $\mathcal{GWI}$-resolution of $N$. Then the sequence $$0\rightarrow \operatorname{Hom}_R(G,N)\rightarrow \operatorname{Hom}_R(G,G^0)\rightarrow \operatorname{Hom}_R(G,G^1)\rightarrow \cdots$$ is exact for any Gorenstein weak projective $R$-module $G$.

$(2)$ Assume that  $M$ is an $R$-module with finite Gorenstein weak projective dimension. Let $\textbf{G}_-=\cdots \rightarrow G_1\rightarrow G_0\rightarrow M\rightarrow 0$ be a proper left $\mathcal{GWP}$-resolution of $M$. Then the sequence $$0\rightarrow \operatorname{Hom}_R(M,G)\rightarrow \operatorname{Hom}_R(G_0,G)\rightarrow \operatorname{Hom}_R(G_1,G)\rightarrow \cdots$$ is exact for any Gorenstein weak injective $R$-module $G$.
\end{lemma}

\begin{proof}
(1) We split the proper resolution $\emph{\textbf{G}}^+$ into short exact sequences. Hence it suffices to show exactness of $\mbox{Hom}_R(G,\emph{\textbf{V}}')$ for any Gorenstein weak projective $R$-module $G$ and any exact sequence $\emph{\textbf{V}}'= 0\rightarrow N\rightarrow G'\rightarrow V'\rightarrow 0$, where $N\rightarrow G'$ is a $\mathcal{GWI}$-preenvelope of some $R$-module $N$ with $Gwid_R(N)<\infty$. By Proposition $\ref{Gorenstein weak injective preenvelope}$, there is a special exact sequence $\emph{\textbf{V}}''= 0\rightarrow N \stackrel{\mu} \rightarrow G'' \stackrel{\nu} \rightarrow V''\rightarrow 0$, where $N\rightarrow G''$ is a $\mathcal{GWI}$-preenvelope of some $R$-module $N$ with $id_R(V'')<\infty$. It is easy to verify that the complexes $\emph{\textbf{V}}'$ and $\emph{\textbf{V}}''$ are homotopy  equivalent, and thus so are the complexes $\mbox{Hom}_R(G,\emph{\textbf{V}}')$ and $\mbox{Hom}_R(G,\emph{\textbf{V}}'')$ for
every Gorenstein weak projective $R$-module $G$. Hence it suffices to show the exactness of
$\mbox{Hom}_R(G,\emph{\textbf{V}}'')$ whenever $G$ is Gorenstein weak projective.

For any Gorenstein weak projective $R$-module $G$, consider an exact sequence
$
  0 \rightarrow G \stackrel{d^0} \rightarrow P^0 \rightarrow G^1 \rightarrow 0 $, where $P^0$ is projective and $G^1$ is Gorenstein weak projective. Assume that $id_R(V'')=n<\infty$. Since $G^1$ is Gorenstein weak projective, there exists an exact sequence $
  0 \rightarrow G^1 \rightarrow P^1 \rightarrow G^2\rightarrow 0 $, where $P^1$ is projective and $G^2$ is Gorenstein weak projective. Hence we have that $\mbox{Ext}^1_R(G^1,V'')\cong \mbox{Ext}^{2}_R(G^2,V'')$. Continue this process, we may choose some Gorenstein weak projective $R$-module $G^{n+1}$ such that $\mbox{Ext}^1_R(G^1,V'')\cong \mbox{Ext}^{2}_R(G^2,V'')\cong \cdots\cong \mbox{Ext}^{n+1}_R(G^{n+1},V'')=0$, and thus get the following exact sequence
  $$
  \xymatrix@C=0.5cm{
    0 \ar[r] & \mbox{Hom}_R(G^1,V'') \ar[r]^{} & \mbox{Hom}_R(P^0,V'') \ar[r]^{} & \mbox{Hom}_R(G,V'') \ar[r] & 0 }
  $$
Therefore, for any $g:G\rightarrow V''$, there exists $h:P^0\rightarrow V''$ such that the following diagram commute:
$$
\xymatrix{
  0 \ar[r] & G \ar[r]^{d^0}\ar[d]^{g} & P^0 \ar[r]^{}\ar@{.>}[ld]^{h} & G^1 \ar[r] & 0 \\
 &V''&&&
  }
$$
Moreover, since $P^0$ is projective, there exists $\rho:P^0\rightarrow G''$  such that the following diagram commute:
$$
\xymatrix{
&&&P^0\ar@{.>}[d]^{h}\ar@{.>}[ld]_{\rho}&G\ar[l]^{f}\ar[ld]^{g}\\
0\ar[r]&N\ar[r]^{\mu}&G''\ar[r]^{\nu}&V''\ar[r]&0
}
$$
This shows that for any $g:G\rightarrow V''$, there exists $\phi=\rho f:G\rightarrow G''$ such that $g=\nu\phi={\nu}_*(\phi)$, and hence
$$
  \xymatrix@C=0.5cm{
    0 \ar[r] & \mbox{Hom}_R(G,N) \ar[r]^{\mu_*} & \mbox{Hom}_R(G,G'') \ar[r]^{\nu_*} & \mbox{Hom}_R(G,V'') \ar[r] & 0 }.
  $$
is exact.

The proof of (2) is similar to that of (1), so we omit it here.
\end{proof}

\begin{definition}
Let $M$ and $N$ be  $R$-modules with $Gwpd_R(M)<\infty$ and $Gwid_R(N)<\infty$. Then we define $$\operatorname{Ext}^n_{\mathcal{GW}}(M,N):=\operatorname{Ext}^n_{\mathcal{GWI}}(M,N)\cong \operatorname{Ext}^n_{\mathcal{GWP}}(M,N), \ n\geq 0$$
and call it the $n$th \emph{Gorenstein weak cohomology}.
\end{definition}

It is natural to compare $\operatorname{Ext}^n_{\mathcal{GW}}(M,N)$ with $\operatorname{Ext}^n_{R}(M,N)$.

\begin{proposition} Let $M$ and $N$ be  $R$-modules. Then

$(1)$ There are natural isomorphisms $\operatorname{Ext}^n_{\mathcal{GWI}}(M,N)\cong \operatorname{Ext}^n_{R}(M,N), \ n\geq 0$ under each of the following conditions
$$
\ id_R(N)<\infty, \mbox{ or }     N\in\operatorname{RightRes}_{_R\mathcal{M}}(\mathcal{GWI})\mbox{ and }pd_R(M)<\infty,
$$

$(2)$ There are natural isomorphisms $\operatorname{Ext}^n_{\mathcal{GWP}}(M,N)\cong \operatorname{Ext}^n_{R}(M,N), \ n\geq 0$ under each of the following conditions
$$
\ pd_R(M)<\infty, \mbox{ or }  M\in\operatorname{LeftRes}_{_R\mathcal{M}}(\mathcal{GWP})\mbox{ and }id_R(N)<\infty,
$$

$(3)$ Assume that $Gwpd_R(M)<\infty$ and $Gwid_R(N)<\infty$. If either $pd_R(M)<\infty$ or $id_R(N)<\infty$, then we have

$$\operatorname{Ext}^n_{\mathcal{GW}}(M,N)\cong \operatorname{Ext}^n_{R}(M,N), \ n\geq 0 .$$
\end{proposition}

\begin{proof}
(1) If $id_R(N)<\infty$, then it is easy to verify that each injective resolution ${\bf I}_N$ of  $N$ is a proper right $\mathcal{GWI}$-resolution of $N$, and hence $$\operatorname{Ext}^n_{\mathcal{GWI}}(M,N)= \mbox{H}^n(\mbox{Hom}_R(M,{\bf I}_N))=\mbox{Ext}^n_R(M,N),\ n\geq 0.$$

If $N\in\mbox{RightRes}_{_R\mathcal{M}}(\mathcal{GWI})\mbox{ and }pd_R(M)=m<\infty$, then for any Gorenstein weak injective $R$-module $G$, there exists an exact sequence $0\rightarrow G'\rightarrow I_{m-1}\rightarrow \cdots \rightarrow I_1\rightarrow I_0\rightarrow G\rightarrow 0$, where each $I_i$ is injective. It follows then that  $\mbox{Ext}^1_R(M,G)\cong\mbox{Ext}^{m+1}_R(M,G')=0$. By \cite[III, Prop. 1.2A]{Ha}, we have $\operatorname{Ext}^n_{\mathcal{GWI}}(M,N)\cong \operatorname{Ext}^n_{R}(M,N), \ n\geq 0$.

The proof of (2) is similar to that of (1), so we omit it here. (3) follows immediately from (1) and (2).
\end{proof}

Following \cite[Thm. 8.2.3 and 8.2.5]{EJ1}, we have the following  long exact sequences induced from  $\mathcal{GWI}$-copure and $\mathcal{GWP}$-pure exact sequences respectively.

\begin{proposition}\label{longsequence}
$(1)$ Assume that $\ell.GwiD(R)<\infty$. If the sequence $0\rightarrow N'\rightarrow N\rightarrow N''\rightarrow 0$ is $\mathcal{GWI}$-copure exact, then we have the following exact sequence
\begin{multline*}
    0\rightarrow \operatorname{Hom}_R(M,N')\rightarrow \operatorname{Hom}_R(M,N)\rightarrow \operatorname{Hom}_R(M,N'')\rightarrow \operatorname{Ext}^1_{\mathcal{GWI}}(M,N') \rightarrow\cdots \\
    \rightarrow \operatorname{Ext}^n_{\mathcal{GWI}}(M,N') \rightarrow \operatorname{Ext}^n_{\mathcal{GWI}}(M,N) \rightarrow \operatorname{Ext}^n_{\mathcal{GWI}}(M,N'') \rightarrow \operatorname{Ext}^{n+1}_{\mathcal{GWI}}(M,N') \rightarrow \cdots
  \end{multline*}
  for any $R$-module $M$.

$(2)$ Assume that $\ell.GwpD(R)<\infty$. If the sequence $0\rightarrow M'\rightarrow M\rightarrow M''\rightarrow 0$ is $\mathcal{GWP}$-pure exact, then we have the following exact sequence
\begin{multline*}
    0\rightarrow \operatorname{Hom}_R(M'',N)\rightarrow \operatorname{Hom}_R(M,N)\rightarrow \operatorname{Hom}_R(M',N)\rightarrow \operatorname{Ext}^1_{\mathcal{GWP}}(M'',N) \rightarrow\cdots \\
    \rightarrow \operatorname{Ext}^n_{\mathcal{GWP}}(M',N) \rightarrow \operatorname{Ext}^n_{\mathcal{GWP}}(M,N) \rightarrow \operatorname{Ext}^n_{\mathcal{GWP}}(M'',N) \rightarrow \operatorname{Ext}^{n+1}_{\mathcal{GWP}}(M',N) \rightarrow \cdots
  \end{multline*}
 for any $R$-module $N$.
\end{proposition}

Now we further give characterizations  of Gorenstein weak injective dimension of $R$-modules in terms of Gorenstein weak right derived functors.

\begin{proposition}\label{characterizations derived functors} Let  $\ell.GwiD(R)<\infty$. Then the following are equivalent:

$(1)$ $Gwid_R(M)\leq n$;

$(2)$ $\operatorname{Ext}^i_R(W,M)=0$ for any weak injective $R$-module $W$ and any $i\geq n+1$;

$(3)$ $\operatorname{Ext}^i_R(\widetilde{W},M)=0$ for any  $R$-module $\widetilde{W}$ with finite weak injective dimension and any $i\geq n+1$;

$(4)$ For every exact sequence $\xymatrix@C=0.5cm{
  0 \ar[r] & M \ar[r]^{} & I^0 \ar[r]^{} & \cdots \ar[r]^{} & I^{n-1} \ar[r]^{} & C^n \ar[r] & 0
  }$, where each $I^i$ is  injective, $C^n$ is Gorenstein weak injective.

$(5)$ For every exact sequence $\xymatrix@C=0.5cm{
  0 \ar[r] & M \ar[r]^{} & G^0 \ar[r]^{} & \cdots \ar[r]^{} & G^{n-1} \ar[r]^{} & V^n \ar[r] & 0
  }$, where each $G^i$ is Gorenstein weak injective, $V^n$ is Gorenstein weak injective.

$(6)$  $\operatorname{Ext}^{n+1}_{\mathcal{GWI}}(M,N)=0$ for any $R$-module $M$;

  $(7)$ $\operatorname{Ext}^i_{\mathcal{GWI}}(M,N)=0$ for any $R$-module $M$ and $i\geq
n+1$.
\end{proposition}

\begin{proof}
(1) $\Leftrightarrow$ (2) $\Leftrightarrow$ $\cdots$ $\Leftrightarrow$ (5) hold by Proposition $\ref{Gweakdim}$.

(1) $\Rightarrow$ (7) Since $Gwid_R(N)\leq n$, there exists a proper right $\mathcal{GWI}$-resolution of $N$:
$$\xymatrix@C=0.5cm{
  0 \ar[r] & N \ar[r]^{} & G^0 \ar[r]^{} & G^1 \ar[r]^{} & \cdots \ar[r]^{} & G^{n-1} \ar[r]^{} & G^n \ar[r] & 0 },$$
and hence $\operatorname{Ext}^i_{\mathcal{GWI}}(M,N)=0$  for any $R$-module $M$ and $i\geq
n+1$.

(7) $\Rightarrow$ (6) is trivial.

(6) $\Rightarrow$ (1) Let $$\xymatrix@C=0.5cm{
  0 \ar[r] & N \ar[r]^{} & G^0 \ar[r]^{} & G^1 \ar[r]^{} & \cdots \ar[r]^{} & G^{n-1} \ar[r]^{} & G^n \ar[r] & \cdots }$$
be a proper right $\mathcal{GWI}$-resolution of $N$, and let $V^0=N$, $V^{1}=\mbox{Coker}(N\rightarrow G^{0})$ and $V^{i}=\mbox{Coker}(G^{i-2}\rightarrow G^{i-1})$ for any $i\geq 2$. Then the exact sequence $0\rightarrow V^{n-1}\rightarrow G^{n-1}\rightarrow V^{n}\rightarrow 0$  is $\mathcal{GWI}$-copure exact. It follows then from Proposition $\ref{longsequence}$ that $\mbox{Ext}^1_{\mathcal{GWI}}(V^{n+1},V^n)\cong \mbox{Ext}^{2}_{\mathcal{GWI}}(V^{n+1},V^{n-1})=0$. Similarly, we have   $\mbox{Ext}^1_{\mathcal{GWI}}(V^{n+1},V^n)\cong \mbox{Ext}^{n+1}_{\mathcal{GWI}}(V^{n+1},N)=0$, and hence
  $$
  \xymatrix@C=0.5cm{
    0 \ar[r] & \mbox{Hom}_R(V^{n+1},V^n) \ar[r]^{} & \mbox{Hom}_R(V^{n+1},D^n) \ar[r]^{} & \mbox{Hom}_R(V^{n+1},V^{n+1}) \ar[r] & 0 }
  $$
is exact. This shows that  $0\rightarrow V^n\rightarrow G^n\rightarrow V^{n+1}\rightarrow 0$ is split. Therefore, $V^n$ is Gorenstein weak injective, as desired.
\end{proof}

Similarly, we have

\begin{proposition} Let $\ell.GwpD(R)<\infty$. Then the following are equivalent:

$(1)$ $Gwpd_R(M)\leq n$;

$(2)$ $\operatorname{Ext}^i_R(M,W)=0$ for any weak flat $R$-module $W$ and any $i\geq n+1$;

$(3)$ $\operatorname{Ext}^i_R(M,\widetilde{W})=0$ for any  $R$-module $\widetilde{W}$ with finite weak flat dimension and any $i\geq n+1$;

$(4)$ For every exact sequence $\xymatrix@C=0.5cm{
  0 \ar[r] & K_n \ar[r]^{} & P_{n-1} \ar[r]^{} & \cdots \ar[r]^{} & P_0 \ar[r]^{} & M \ar[r] & 0
  }$, where each $P_i$ is  projective, $K_n$ is Gorenstein weak projective.

$(5)$ For every exact sequence $\xymatrix@C=0.5cm{
  0 \ar[r] & K_n' \ar[r]^{} & G_{n-1} \ar[r]^{} & \cdots \ar[r]^{} & G_0 \ar[r]^{} & M \ar[r] & 0
  }$, where each $G_i$ is Gorenstein weak projective, $K_n'$ is Gorenstein weak projective.

$(6)$  $\operatorname{Ext}^{n+1}_{\mathcal{GWP}}(M,N)=0$ for any $R$-module $N$;

  $(7)$ $\operatorname{Ext}^i_{\mathcal{GWP}}(M,N)=0$ for any $R$-module $N$ and $i\geq
n+1$.
\end{proposition}

\section{Tate derived functors with respect to Gorenstein weak  modules}

 In this section we continue to investigate
another derived functor, $\widehat{\operatorname{Ext}}^n_{\mathcal{GW}}(-,-)$,  which connects the usual right derived functor ${\operatorname{Ext}}^n_{R}(-,-)$ with the Gorenstein weak  right derived functor ${\operatorname{Ext}}^n_{\mathcal{GW}}(-,-)$.

We first introduce the following related notions.

\begin{definition}\label{def11} A \emph{$\mathcal{WI}$-pure exact complex of injective
$R$-modules} is an exact complex of injective
$R$-modules
$$\textbf{I}= \ \     \xymatrix@C=0.5cm{
  \cdots \ar[r] & I_1 \ar[r]^{d_1} & I_0 \ar[r]^{d_0} & I^0 \ar[r]^{d^0} & I^1 \ar[r]^{d^1} & \cdots }$$
  such that the complex $\operatorname{Hom}_R(W,\textbf{I})$ is exact for any
weak injective  $R$-module $W$.\end{definition}

Note that an $R$-module $M$ is Gorenstein weak injective if and only if there is a
$\mathcal{WI}$-pure exact complex $\textbf{I}$ of injective  $R$-modules
such that $M\cong \mbox{Coker}(I_1\rightarrow I_0)$. Moreover, if
there is a $\mathcal{WI}$-pure exact complex $\textbf{I}$ of injective  $R$-modules, then each kernel, cokernel and image in $\textbf{I}$
are Gorenstein weak injective.

\begin{definition}\label{def3.2} Let $M$ be an $R$-module. A \emph{$\mathcal{WI}$-pure  Tate injective
resolution} of $M$ is a diagram $\xymatrix@C=0.5cm{
  M \ar[r] & \textbf{E} \ar[r]^{u} & \textbf{T}  }$, where $\textbf{E}$ is a deleted injective
  resolution of $M$ and $\textbf{T}$ is a $\mathcal{WI}$-pure exact exact complex of
  injective
$R$-modules and $u$ is a morphism of complexes  such that $u^n$ is isomorphic for  $n\gg 0$.\end{definition}

 For example, if $M$ is an $R$-module  with $id_R(M)<\infty$, then the zero
 complex is a $\mathcal{WI}$-pure  Tate injective resolution of $M$, and  if $M$ is a Gorenstein weak
 injective $R$-module such that there is a $\operatorname{Hom}_R(\mathcal{WI},-)$-exact exact complex $\textbf{I}= \ \xymatrix@C=0.5cm{
   \cdots \ar[r] & I_1 \ar[r]^{} & I_0 \ar[r]^{} & I^{0}  \ar[r] & \cdots }$
 and $M\cong \mbox{Coker}(I_1\rightarrow
   I_0)$, then $\textbf{I}$ is a $\mathcal{WI}$-pure  Tate injective resolution
   of $M$, in this case $n=0$.

\begin{lemma}\label{lem3} Let $M$ be an $R$-module. Then
$Gwid_R(M)<\infty$ if and only if $M$ has a $\mathcal{WI}$-pure  Tate injective
resolution.\end{lemma}

\begin{proof} Assume that $Gwid_R(M)=n<\infty$.  Consider an injective resolution of $M$: $0\rightarrow M\rightarrow I^0\rightarrow I^1\rightarrow\cdots$. Let $V^n=\mbox{Coker}(I^{n-2}\rightarrow I^{n-1})$. Then $V^n$ is Gorenstein weak injective, and hence we have the following commutative diagram:
$$
\xymatrix@=0.5cm{
0\ar[r]&N\ar[r]&I^0\ar[r]\ar@{.>}[dd]&I^1\ar[r]\ar@{.>}[dd]&\cdots\ar[r]&I^{n-1}\ar[rr]\ar@{.>}[dd]\ar[rd]&&E^n\ar[r]\ar@{=}[dd]&\cdots\\
&&&&&&V^n\ar[ru]\ar[rd]&&\\
\cdots\ar[r]&E^{-1}\ar[r]&E^0\ar[r]&E^1\ar[r]&\cdots\ar[r]&E^{n-1}\ar[rr]\ar[ru]&&E^n\ar[r]&\cdots
}
$$
where the bottom row is a $\mbox{Hom}_R(\mathcal{WI},-)$-exact exact sequence and each $E^i$ is injective. Thus this diagram is a $\mathcal{WI}$-pure  Tate injective
resolution of $M$.

Conversely, suppose that $M$ has a $\mathcal{WI}$-pure  Tate injective
resolution. Without loss of generality, we may assume that it is a diagram as shown in the above. Then $V^n$ is Gorenstein weak injective, and hence $Gwid_R(M)\leq n<\infty$.
\end{proof}

\begin{definition}\label{relative Tate cohomology} If an $R$-module $M$ has a $\mathcal{WI}$-pure  Tate injective
resolution $
  M \rightarrow \textbf{E} \rightarrow \textbf{T}$, then we define the relative Tate cohomology of $M$ with coefficient in an $R$-module $N$ as
$$\widehat{\mbox{Ext}}^i_{\mathcal{GWI}}(N,M)=\mbox{H}^i(\mbox{Hom}_R(N,\textbf{T})).$$
\end{definition}

We first claim that the above definition doesn't depend on the choice
of $\mathcal{WI}$-pure Tate injective resolutions of $M$. Indeed, assume that
$\xymatrix@C=0.5cm{
  M \ar[r] & \textbf{E} \ar[r]^{u} & \textbf{T}  }$ and $\xymatrix@C=0.5cm{
  M \ar[r] & \textbf{E}' \ar[r]^{v} & \textbf{T}'  }$ are two $\mathcal{WI}$-pure Tate injective resolutions of
  $M$ such that $u^{n'}$ is isomorphic for $n'\gg 0$ and $v^{n''}$ is isomorphic for $n''\gg 0$. Let
   $n=\mbox{max}\{n',n''\}$. If $i>n$, then
   $\mbox{H}^i(\mbox{Hom}_R(N,\textbf{T}))\cong\mbox{Ext}^i_R(N,M)\cong\mbox{H}^i(\mbox{Hom}_R(N,\textbf{T}'))$.
   If $i\leq n$, we consider an exact sequence $
     0 \rightarrow N \rightarrow W \rightarrow V^0 \rightarrow 0 $
     with $W$ a weak injective preenvelope of $N$, then we have the
     following exact sequence of complexes $$\xymatrix@C=0.5cm{
       0 \ar[r] & \mbox{Hom}_R(V^0,\textbf{T}) \ar[r]^{} & \mbox{Hom}_R(W,\textbf{T}) \ar[r]^{} & \mbox{Hom}_R(N,\textbf{T}) \ar[r] & 0
       },$$ which induces a long exact sequence of $R$-modules
\begin{multline*}
  \cdots \longrightarrow \operatorname{H}^i(\mbox{Hom}_R(W,\textbf{T})) \longrightarrow \operatorname{H}^i(\mbox{Hom}_R(N,\textbf{T})) \\
  \longrightarrow \operatorname{H}^{i+1}(\mbox{Hom}_R(V^0,\textbf{T}))  \longrightarrow
         \operatorname{H}^{i+1}(\mbox{Hom}_R(W,\textbf{T})) \longrightarrow \cdots .
\end{multline*}
By Definition $\ref{def11}$,
$\operatorname{H}^i(\mbox{Hom}_R(W,\textbf{T}))=0=\operatorname{H}^{i+1}(\mbox{Hom}_R(W,\textbf{T}))$, and hence
$\operatorname{H}^i(\mbox{Hom}_R(N,\textbf{T}))\cong
\operatorname{H}^{i+1}(\mbox{Hom}_R(V^0,\textbf{T})).$ Repeating this process, we may
find $V^j$ such that $\operatorname{H}^i(\mbox{Hom}_R(N,\textbf{T})) \cong
\operatorname{H}^{i+j+1}(\mbox{Hom}_R(V^j,\textbf{T})) \mbox{ and }i+j+1>n.$ Hence
$\operatorname{H}^i(\mbox{Hom}_R(N,\textbf{T}))\cong\mbox{Ext}^{i+j+1}_R(V^{j},M)$.
Similarly, we also have
$\operatorname{H}^i(\mbox{Hom}_R(N,\textbf{T}'))\cong\mbox{Ext}^{i+j+1}_R(V^{j},M)$.

\begin{proposition} Let $M$ be an $R$-module with $Gwid_R(M)<\infty$. For an exact sequence $\xymatrix@C=0.5cm{
  0 \ar[r] & A \ar[r]^{} & B \ar[r]^{} & C \ar[r] & 0 }$ of
  $R$-modules,  we have the following exact sequence
\begin{multline*}
  \cdots \longrightarrow\widehat{\operatorname{Ext}}^{i-1}_{\mathcal{GWI}}(A,M) \longrightarrow \widehat{\operatorname{Ext}}^i_{\mathcal{GWI}}(C,M) \longrightarrow \widehat{\operatorname{Ext}}^i_{\mathcal{GWI}}(B,M) \\
  \longrightarrow \widehat{\operatorname{Ext}}^i_{\mathcal{GWI}}(A,M) \longrightarrow \widehat{\operatorname{Ext}}^{i+1}_{\mathcal{GWI}}(C,M) \longrightarrow \cdots
\end{multline*}
for any $i\in \mathbb{Z}$.
\end{proposition}

\begin{proof} By Lemma $\ref{lem3}$, $M$ has a  $\mathcal{WI}$-pure Tate injective resolution
$\xymatrix@C=0.5cm{
  M \ar[r] & \textbf{E} \ar[r]^{u} & \textbf{T}  }$. Since each term of $\textbf{T}$ is injective,  we
have the following exact sequence of complexes
$
  0 \rightarrow \operatorname{Hom}_R(C,\textbf{T}) \rightarrow \operatorname{Hom}_R(B,\textbf{T}) \rightarrow \operatorname{Hom}_R(A,\textbf{T}) \rightarrow 0 ,$
 which induces a long exact sequence
\begin{multline*}
  \cdots \longrightarrow \mbox{H}^{i-1}(\mbox{Hom}_R(A,\textbf{T}))\longrightarrow \mbox{H}^i(\mbox{Hom}_R(C,\textbf{T})) \longrightarrow \mbox{H}^i(\mbox{Hom}_R(B,\textbf{T}))\\
  \longrightarrow\mbox{H}^i(\mbox{Hom}_R(A,\textbf{T})) \longrightarrow \mbox{H}^{i+1}(\mbox{Hom}_R(C,\textbf{T})) \longrightarrow \cdots,
\end{multline*}
as required.\end{proof}

The following theorem shows the case of vanishing of relative Tate
cohomology defined as in Definition $\ref{relative Tate cohomology}$.

\begin{theorem}\label{thm3.6} Let $M$ be an $R$-module with $Gwid_R(M)=n<\infty$. The following are
equivalent:

$(1)$ $id_R(M)\leq n$;

$(2)$ $id_R(M)<\infty$;

$(3)$ $\widehat{\operatorname{Ext}}^i_{\mathcal{GWI}}(N,M)=0$ for any $R$-module $N$ and any
$i\in \mathbb{Z}$;

$(4)$ $\widehat{\operatorname{Ext}}^i_{\mathcal{GWI}}(R/I,M)=0$ for any left ideal $I$ of
$R$ and any $i\in \mathbb{Z}$.\end{theorem}

\begin{proof} (1)$\Rightarrow$(2) and  (3)$\Rightarrow$(4) are
trivial.

(2)$\Rightarrow$(3). Since $id_R(M)<\infty$, we may take a  $\mathcal{WI}$-pure Tate
injective resolution of $M$ to be the zero complex, and thus
$\widehat{\mbox{Ext}}^i_{\mathcal{GWI}}(N,M)=0$ for any $N\in R$-Mod and any
$i\in \mathbb{Z}$.

(4)$\Rightarrow$(1). We use  induction on $n=Gwid_R(M)<\infty$. If $Gwid_R(M)=0$, then
$\mbox{Ext}^1_R(R/I,M)\cong\widehat{\mbox{Ext}}^1_{\mathcal{GWI}}(R/I,M)=0$ for
any left ideal $I$ of $R$, which implies that $M$ is injective, i.e.
$id_R(M)=0$. Now we assume that $Gwid_R(M)>0$, and let $\xymatrix@C=0.5cm{
  M \ar[r] & \textbf{E} \ar[r]^{u} & \textbf{T}  }$ be a $\mathcal{WI}$-pure Tate injective resolution of
  $M$ and $M'=\mbox{Coker}(M\rightarrow E^0)$. Then we have an exact sequence $
    0 \rightarrow M \rightarrow E^0 \rightarrow M' \rightarrow 0 $  with $E^0$ injective. Moreover,
     $Gwid_R(M')\leq n-1$ and  $\textbf{T}[-1]$ is a
  weak Tate injective resolution of $M'$. This implies that
  $\widehat{\mbox{Ext}}^{i}_{\mathcal{GWI}}(N,M')\cong\widehat{\mbox{Ext}}^{i-1}_{\mathcal{GWI}}(N,M)$
  for any $N\in R$-Mod and any $i\in \mathbb{Z}$. In particular,
  $\widehat{\mbox{Ext}}^{i}_{\mathcal{GWI}}(R/I,M')\cong\widehat{\mbox{Ext}}^{i-1}_{\mathcal{GWI}}(R/I,M)=0$
  for any left ideal $I$ of
$R$ and  any $i\in\mathbb{Z}$. This implies $id_R(M')\leq n-1$ by the
induction hypothesis, and hence $id_R(M)\leq n$.
\end{proof}

 We also have the following long exact
sequence with respect to the usual  cohomology, the Gorenstein weak cohomology and the relative
Tate cohomology, which is similar to that in \cite[Sec. 4]{Ia}:

\begin{lemma}\label{lem3.7} Let $M$ be an $R$-module with $Gwid_R(M)<\infty$.
Then we have a long exact sequence
$$\xymatrix@C=0.4cm{
  0 \ar[r] & \operatorname{Ext}^1_{\mathcal{GWI}}(N,M) \ar[r]^{} & \operatorname{Ext}^1_R(N,M) \ar[r]^{} & \widehat{\operatorname{Ext}}^{1}_{\mathcal{GWI}}(N,M) \ar[r]^{} & \operatorname{Ext}^2_{\mathcal{GWI}}(N,M) \ar[r] &
  \cdots }$$ for any $R$-module $N$.\end{lemma}

Both this  and the following proposition show that the relative Tate cohomology
measures the distance between the cohomology and the Gorenstein weak
cohomology.

\begin{proposition}\label{prop3.8}  Let $M$ and $N$ be  $R$-modules with $Gwid_R(M)=n<\infty$.
If $id_R(M)<\infty$, then the natural transformation $\operatorname{Ext}^i_{\mathcal{GWI}}(N,M)
\rightarrow \operatorname{Ext}^i_R(N,M)$ is a natural isomorphism for any $0\leq
i\leq n$, and $\operatorname{Ext}^i_R(N,M)=0$ for any $i>n$.\end{proposition}

\begin{proof} If $0<
i\leq n$, then it follows immediately from Theorem $\ref{thm3.6}$ and Lemma
$\ref{lem3.7}$. Moreover, $\operatorname{Ext}^0_{\mathcal{GWI}}(N,M) \cong \operatorname{Hom}_R(N,M)\cong
\operatorname{Ext}^0_R(N,M)$. So the assertion  holds for $0\leq i\leq n$.
Furthermore, $\operatorname{Ext}^i_{\mathcal{GWI}}(N,M)=0=\widehat{\operatorname{Ext}}^{i}_{\mathcal{GWI}}(N,M)$
whenever $i>n$, which implies that $\operatorname{Ext}^i_R(N,M)=0$ for all $i>n$ by
the exact sequence of Lemma $\ref{lem3.7}$.\end{proof}

\begin{lemma} Let $M$ and $N$ be $R$-modules with $id_R(N)<\infty$ or $pd_R(N)<\infty$. If  $M$
admits a $\mathcal{WI}$-pure Tate injective resolution $\xymatrix@C=0.5cm{
  M \ar[r] & \operatorname{\bf E} \ar[r]^{u} & \operatorname{\bf T}  }$. Then  $\widehat{\operatorname{Ext}}^{i}_{\mathcal{GWI}}(N,M)=0$ for any $i\in
\mathbb{Z}$.\end{lemma}

\begin{proof}
We only prove the case $id_R(N)<\infty$, the proof of the case $pd_R(N)<\infty$ is similar.
To end it, it suffices to prove that
the complex $\mbox{Hom}_R(N,\textbf{T})$ is exact by Definition $\ref{relative Tate cohomology}$.

We use  induction on $n=id_R(N)<\infty$. If $n=0$, then
$\mbox{Hom}_R(N,\textbf{T})$ is exact. Now we assume that $n>0$, and
consider an exact sequence $\xymatrix@C=0.5cm{
  0 \ar[r] & N \ar[r]^{} & E \ar[r]^{} & N' \ar[r] & 0 }$ with $E$
  injective and thus $id_R(N')=n-1$. Then we have the following exact
  sequence of complexes $$\xymatrix@C=0.5cm{
    0 \ar[r] & \mbox{Hom}_R(N',\textbf{T}) \ar[r]^{} & \mbox{Hom}_R(E,\textbf{T}) \ar[r]^{} & \mbox{Hom}_R(N,\textbf{T}) \ar[r] & 0
    }.$$ Note that the complex $\mbox{Hom}_R(E,\textbf{T})$ is exact and the complex $\mbox{Hom}_R(N',\textbf{T})$
    is also exact by the induction hypothesis, which implies that the complex
    $\mbox{Hom}_R(N,\textbf{T})$ is exact, as desired.\end{proof}

By this lemma, we can refine Proposition $\ref{prop3.8}$ as follows.

\begin{proposition} Let $M$ and $N$ be  $R$-modules with $Gwid_R(M)=n<\infty$.
If $id_R(M)<\infty$ or $id_R(N)<\infty$, then the natural transformation
$\operatorname{Ext}^i_{\mathcal{GWI}}(N,M) \rightarrow \operatorname{Ext}^i_R(N,M)$ is a natural isomorphism
for any $0\leq i\leq n$, and $\operatorname{Ext}^i_R(N,M)=0$ for any
$i>n$.\end{proposition}

Similar to Definitions $\ref{def11}$, $\ref{def3.2}$ and $\ref{relative Tate cohomology}$, we give the following definitions.

\begin{definition}\label{def311} A \emph{$\mathcal{WF}$-copure exact complex of projective
$R$-modules} is an exact complex of projective
$R$-modules
$$\textbf{P}= \ \     \xymatrix@C=0.5cm{
  \cdots \ar[r] & P_1 \ar[r]^{d_1} & P_0 \ar[r]^{d_0} & P^0 \ar[r]^{d^0} & P^1 \ar[r]^{d^1} & \cdots }$$
  such that the complex $\operatorname{Hom}_R(\textbf{P},W)$ is exact for any
weak flat  $R$-module $W$.\end{definition}

Note that $M\in R$-Mod is Gorenstein weak projective if and only if there is a
$\mathcal{WF}$-copure exact complex $\textbf{P}$ of projective  $R$-modules
such that $M\cong \mbox{Coker}(P_1\rightarrow P_0)$. Moreover, if
there is a $\mathcal{WF}$-copure exact complex $\textbf{P}$ of projective  $R$-modules, then each kernel, cokernel and image in $\textbf{P}$
are Gorenstein weak projective.

\begin{definition}\label{def33.2} Let $M$ be an $R$-module. A \emph{$\mathcal{WF}$-copure  Tate projective
resolution} of $M$ is a diagram $\xymatrix@C=0.5cm{
 \textbf{T} \ar[r] & \textbf{P} \ar[r]^{u} &M }$, where $\textbf{P}$ is a deleted projective
  resolution of $M$ and $\textbf{T}$ is a $\mathcal{WF}$-copure exact  complex of
  projective
$R$-modules and $u$ is a morphism of complexes  such that $u_n$ is isomorphic for  $n\gg 0$.\end{definition}

 For example, if $M\in R$-Mod with $pd_R(M)<\infty$, then the zero
 complex is a $\mathcal{WF}$-copure  Tate projective resolution of $M$, and  if $M\in R$-Mod is a Gorenstein weak
 projective module such that there is a $\operatorname{Hom}_R(-,\mathcal{WF})$-exact exact complex $\textbf{P}=
   \cdots \rightarrow P_1 \rightarrow P_0 \rightarrow P^{0}  \rightarrow \cdots $
 and $M\cong \mbox{Coker}(P_1\rightarrow
   P_0)$, then $\textbf{P}$ is a $\mathcal{WF}$-copure  Tate projective resolution
   of $M$.

\begin{lemma}\label{lem33} Let $M$ be an $R$-module. Then
$Gwpd_R(M)<\infty$ if and only if $M$ has a $\mathcal{WF}$-copure  Tate projective
resolution.\end{lemma}

\begin{definition}\label{3relative Tate cohomology} If $M\in R\operatorname{-Mod}$ has a $\mathcal{WF}$-copure  Tate projective
resolution $
  \textbf{T} \rightarrow \textbf{P} \rightarrow M$, then we define the relative Tate cohomology of $M$ with coefficient in an $R$-module $N$ as
$$\widehat{\mbox{Ext}}^i_{\mathcal{GWP}}(M,N)=\mbox{H}^i(\mbox{Hom}_R(\textbf{T},N)).$$
\end{definition}

As a similar argument in the above, we may show that this definition doesn't depend on the choice
of $\mathcal{WF}$-copure Tate projective resolutions of $M$.

It is well-known that  $\mbox{Ext}^n_R(M,N)$ can be compute by a projective resolution of $M$ or a injective resolution of $N$. It is natural to ask that whether $\widehat{\mbox{Ext}}^n_{\mathcal{GWP}}(M,N)=\widehat{\mbox{Ext}}^n_{\mathcal{GWI}}(M,N)$ hold or not? The following theorem gives an affirmative answer.

\begin{theorem}\label{Tate balance}
Let $M$ and $N$ be $R$-modules with $Gwpd_R(M)<\infty$ and $Gwid_R(M)<\infty$. Then $\widehat{\operatorname{Ext}}^i_{\mathcal{GWP}}(M,N)=\widehat{\operatorname{Ext}}^i_{\mathcal{GWI}}(M,N)$ for any $i\in
\mathbb{Z}$.
\end{theorem}

\begin{proof}
We use induction on $n=Dpd_R(M)<\infty$. If $M$ is Gorenstein weak projective, then $M$ admits a $\mathcal{WF}$-copure  Tate projective
resolution  $\xymatrix@C=0.5cm{
  \textbf{T} \ar[r]^{u} & \textbf{P} \ar[r]^{\pi} &M }$, where $T_i=P_i$ and  $u_i=\mbox{Id}_{P_i}$ for any $i\geq 0$. Thus $\widehat{\mbox{Ext}}^i_{\mathcal{GWP}}(M,N)\cong \mbox{Ext}^i_R(M,N), i\geq 1$. Note that $\mbox{Ext}^i_{\mathcal{GWI}}(M,N)\cong\mbox{Ext}^i_{\mathcal{GWP}}(M,N)=0$ for any $i\geq 1$, and hence
  $\widehat{\mbox{Ext}}^i_{\mathcal{GWI}}(M,N)\cong \mbox{Ext}^i_R(M,N)$ by Lemma $\ref{lem3.7}$.
  Therefore, $\widehat{\operatorname{Ext}}^i_{\mathcal{GWP}}(M,N)=\widehat{\operatorname{Ext}}^i_{\mathcal{GWI}}(M,N)$ for any  $i\geq 1$. Now we see the case $i\leq 0$. Since $N$ has finite Gorenstein weak injective dimension, we may take a $\mathcal{WI}$-pure  Tate injective
resolution
  $\xymatrix@C=0.5cm{
 N \ar[r] & \textbf{E} \ar[r] &\textbf{T}' }$. Assume that $\textbf{T}$ in the above resolution is of the form
 $$\xymatrix@C=0.5cm{ \cdots\ar[r]&P_1\ar[r]^{d_1}&P_0\ar[r]^{d_0}&P_{-1}\ar[r]^{d_{-1}}&P_{-2}\ar[r]&\cdots}.$$
 By the definition of $\mathcal{WF}$-copure  Tate projective
resolution, each $M_i:=\mbox{Im}d_i$ is Gorenstein weak projective. Let $M_{-1}=\mbox{Im}d_{-1}$. Then we have an exact sequence $$\xymatrix@C=0.5cm{
   0 \ar[r] & M \ar[r]^{} & P_{-1} \ar[r]^{} & M_{-1} \ar[r] & 0 }.$$ Since each term of $\textbf{T}'$ is injective, we have the following exact sequence of complexes
 $$\xymatrix@C=0.5cm{
   0 \ar[r] & \mbox{Hom}_R(M_{-1},\textbf{T}') \ar[r]^{} & \mbox{Hom}_R(P_{-1},\textbf{T}') \ar[r]^{} &
   \mbox{Hom}_R(M,\textbf{T}')\ar[r] & 0 },$$
which induced the following exact sequence
\begin{multline*}
 \cdots\rightarrow\mbox{H}^i(\mbox{Hom}_R(P_{-1},\textbf{T}'))\rightarrow\mbox{H}^i(\mbox{Hom}_R(M,\textbf{T}'))\rightarrow  \\
  \mbox{H}^{i+1}(\mbox{Hom}_R(M_{-1},\textbf{T}'))\rightarrow
\mbox{H}^{i+1}(\mbox{Hom}_R(P_{-1},\textbf{T}')\rightarrow\cdots.
\end{multline*}
It is obvious that $\mbox{H}^i(\mbox{Hom}_R(P_{-1},\textbf{T}'))=0=\mbox{H}^{i+1}(\mbox{Hom}_R(P_{-1},\textbf{T}')$. So  $\mbox{H}^i(\mbox{Hom}_R(M,\textbf{T}'))\cong
\mbox{H}^{i+1}(\mbox{Hom}_R(M_{-1},\textbf{T}'))$, that is,
$\widehat{\mbox{Ext}}^i_{\mathcal{GWI}}(M,N)\cong\widehat{\mbox{Ext}}^{i+1}_{\mathcal{GWI}}(M_{-1},N)$.
Repeating this process, we may get that $\widehat{\mbox{Ext}}^i_{\mathcal{GWI}}(M,N)\cong\widehat{\mbox{Ext}}^{1}_{\mathcal{GWI}}(M_{i-1},N)$.
On the other hand, it is obvious to verify that $
  \textbf{T}[-1] \rightarrow \textbf{P}[-1] \rightarrow M_{-1} $ is a $\mathcal{WF}$-copure  Tate projective
resolution of $M_{-1}$, and hence  $\widehat{\mbox{Ext}}^i_{\mathcal{GWP}}(M_{-1},N)\cong \widehat{\mbox{Ext}}^{i-1}_{\mathcal{GWP}}(M,N)$. By a similar argument as in the above, we have $\widehat{\mbox{Ext}}^i_{\mathcal{GWP}}(M,N)\cong \widehat{\mbox{Ext}}^{1}_{\mathcal{GWP}}(M_{i-1},N)$, and hence $$\widehat{\mbox{Ext}}^i_{\mathcal{GWI}}(M,N)\cong \widehat{\mbox{Ext}}^{1}_{\mathcal{GWI}}(M_{i-1},N)\cong \widehat{\mbox{Ext}}^1_{\mathcal{GWP}}(M_{i-1},N)\cong \widehat{\mbox{Ext}}^{i}_{\mathcal{GWP}}(M,N),\ i\leq 0.$$
Therefore, we have $\widehat{\mbox{Ext}}^i_{\mathcal{GWP}}(M,N)=\widehat{\mbox{Ext}}^i_{\mathcal{GWI}}(M,N),
i\in \mathbb{Z}$, whenever $M$ is Gorenstein weak projective.

Assume that the assertion holds for the case $n-1$. Let $Gwpd_R(M)=n$ and consider a $\mathcal{WF}$-copure  Tate projective
resolution of $M$ as follows:
$$\xymatrix@=0.5cm{
M&=&&&&M&&\\
\textbf{P}\ar[u]&=&\cdots\ar[r]&P_2\ar[r]&P_1\ar[r]&P_0\ar[u]\\
\textbf{T}\ar[u]&=&\cdots\ar[r]&T_2\ar[r]\ar[u]&T_1\ar[r]\ar[u]&T_0\ar[r]\ar[u]&T_{-1}\ar[r]&\cdots.
  }$$
Let $M_1=\mbox{Ker}(P_0\rightarrow M)$. Then we have an exact sequence $\xymatrix@C=0.5cm{
  0 \ar[r] & M_1 \ar[r]^{} & P_0 \ar[r]^{} & M \ar[r] & 0 }$, and it is easy to verify that the following diagram
$$\xymatrix@=0.5cm{
M_1&=&&&&M_1&&\\
\textbf{P}[1]\ar[u]&=&\cdots\ar[r]&P_3\ar[r]&P_2\ar[r]&P_1\ar[u]\\
\textbf{T}[1]\ar[u]&=&\cdots\ar[r]&T_3\ar[r]\ar[u]&T_2\ar[r]\ar[u]&T_1\ar[r]\ar[u]&T_{0}\ar[r]&T_{-1}\ar[r]&\cdots
  }$$
is a $\mathcal{WF}$-copure  Tate projective
resolution of $M_1$. Thus $\widehat{\mbox{Ext}}^i_{\mathcal{GWP}}(M,N)=\widehat{\mbox{Ext}}^{i-1}_{\mathcal{GWP}}(M_1,N)$ for any  $i\in
\mathbb{Z}$.

On the other hand, with a similar argument to the first step, we have the following exact sequence of complexes
 $$\xymatrix@C=0.5cm{
   0 \ar[r] & \mbox{Hom}_R(M,\textbf{T}') \ar[r]^{} & \mbox{Hom}_R(P_{0},\textbf{T}') \ar[r]^{} &
   \mbox{Hom}_R(M_1,\textbf{T}')\ar[r] & 0 }$$
which induced the following exact sequence
\begin{multline*}
  \cdots\rightarrow\mbox{H}^{i-1}(\mbox{Hom}_R(P_{0},\textbf{T}'))\rightarrow\mbox{H}^{i-1}(\mbox{Hom}_R(M_1,\textbf{T}'))\rightarrow\\
  \mbox{H}^{i}(\mbox{Hom}_R(M,\textbf{T}'))\rightarrow
\mbox{H}^{i}(\mbox{Hom}_R(P_{0},\textbf{T}')\rightarrow\cdots.
\end{multline*}
It is obvious that $\mbox{H}^{i-1}(\mbox{Hom}_R(P_{0},\textbf{T}'))=0=\mbox{H}^{i}(\mbox{Hom}_R(P_{0},\textbf{T}')$. Thus
$\mbox{H}^i(\mbox{Hom}_R(M_1,\textbf{T}'))\cong
\mbox{H}^{i+1}(\mbox{Hom}_R(M,\textbf{T}'))$, that is,
$\widehat{\mbox{Ext}}^{i-1}_{\mathcal{GWI}}(M_1,N)\cong\widehat{\mbox{Ext}}^{i}_{\mathcal{GWI}}(M,N)$.
Since  $Gwpd_R(M_1)=n-1$, we have  $\widehat{\mbox{Ext}}^{i-1}_{\mathcal{GWP}}(M_1,N)\cong
\widehat{\mbox{Ext}}^{i-1}_{\mathcal{GWI}}(M_1,N)$ by the induction hypothesis. Therefore, $$\widehat{\mbox{Ext}}^i_{\mathcal{GWP}}(M,N)\cong
\widehat{\mbox{Ext}}^i_{\mathcal{GWI}}(M,N)$$ for any $ i\in \mathbb{Z}$, as desired.
\end{proof}

\begin{definition}
Let $M$ and $N$ be  $R$-modules with $Gwpd_R(M)<\infty$ and $Gwid_R(N)<\infty$. Then we define $$\widehat{\operatorname{Ext}}^n_{\mathcal{GW}}(M,N):=\widehat{\operatorname{Ext}}^n_{\mathcal{GWI}}(M,N)\cong \widehat{\operatorname{Ext}}^n_{\mathcal{GWP}}(M,N), \ n\geq \mathbb{Z}$$
and call it the $n$th \emph{Gorenstein weak Tate cohomology}.
\end{definition}

Following Lemma $\ref{lem3.7}$, we have

\begin{proposition} Let $M$ and $N$ be  $R$-modules with $Gwid_R(M)<\infty$ and $Gwid_R(N)<\infty$.
Then we have a long exact sequence
$$\xymatrix@C=0.4cm{
  0 \ar[r] & \operatorname{Ext}^1_{\mathcal{GW}}(N,M) \ar[r]^{} & \operatorname{Ext}^1_R(N,M) \ar[r]^{} & \widehat{\operatorname{Ext}}^{1}_{\mathcal{GW}}(N,M) \ar[r]^{} & \operatorname{Ext}^2_{\mathcal{GW}}(N,M) \ar[r] &
  \cdots }.$$\end{proposition}


\end{document}